\documentclass[11pt]{article}
\usepackage{amssymb,amsthm,amsfonts,amssymb,euscript,mathrsfs
}
\usepackage{amsfonts, amsmath}
\usepackage{amssymb}
\usepackage{amsxtra}
\usepackage{amscd}
\usepackage{latexsym}
\usepackage{exscale}
\usepackage[latin1]{inputenc}
\usepackage[cyr]{aeguill}
\usepackage[french]{babel}
\usepackage[T1]{fontenc}
\numberwithin{equation}{section}

\title{ \bf Stabilité des sous-algèbres biparaboliques des algèbres de Lie simples.}
\date{}
\author{Kais Ammari}

\makeatletter

\@addtoreset{equation}{section}
\makeatother

\newtheorem{theo}[subsection]{Th\'eor\`eme}
\newtheorem{pr}[subsection]{Proposition}
\newtheorem{lem}[subsection]{Lemme}
\newtheorem{defi}[subsection]{D\'efinition}
\newtheorem{co}[subsection]{Corollaire}
\newtheorem{rema}[subsection]{Remarque}

\newtheorem{conj}[subsection]{Conjecture reformulée}

\def\ind{\rm ind\,}
\def\rang{\rm rang\,}
\def\text{ \rm }

\input{Dynkin.sty}

\begin{document}.
\begin{center} { $\textbf{ Stabilité des sous-algèbres biparaboliques des algèbres de Lie simples. }$ } \end{center}
\
\begin{center} { Kais Ammari } \end{center}
\
\

\begin{abstract}
Soit $\mathbb{K}$ un corps algébriquement clos de caractéristique nulle. Dans cet article, nous montrons qu'une sous-algèbre biparabolique d'une algèbre de Lie simple  est stable si et seulement si elle est quasi-réductive. Par conséquent et compte tenu des résultats de \cite{Ammari} et \cite{Ammari2}, on donne une réponse positive à l'assertion $\text{(ii)}$ de la conjecture $5.6$ de \cite{Panyushev2005}.\\
\\\textbf{Abstract}. Let $\mathbb{K}$ be an algebraically closed field of characteristic 0. In this paper, we prove the equivalence between stability and quasi-reductivity for biparabolic subalgebras of simple Lie algebras. Therefore,
we give a positive answer to the assertion $\text(ii)$ of the conjecture $(5.6)$ in \cite{Panyushev2005}.\\
\end{abstract}
\section{Introduction}
Dans toute la suite, $\mathbb{K}$ désigne un corps algébriquement clos de caractéristique nulle. Les alg\`{e}bres de Lie consid\'{e}r\'{e}es sont d\'{e}finies et de dimension finie sur $\mathbb{K}$. Soit $\mathfrak{g}$ une alg\`{e}bre de Lie alg\'{e}brique et $G$ un groupe de Lie alg\'{e}brique affine connexe d'alg\`{e}bre de Lie $\mathfrak{g}$. On munit $\mathfrak{g}^{\ast}$, l'espace dual de $\mathfrak{g}$, des actions coadjointes de $\mathfrak{g}$ et de $G$. \'{E}tant donn\'{e}e une forme lin\'{e}aire $g\in \mathfrak{g}^{\ast}$, on note $\mathfrak{g}(g)$ son stabilisateur dans $\mathfrak{g}$. On identifie $\mathfrak{g}(g)/\mathfrak{z}$, o\`{u} $\mathfrak{z}$ d\'{e}signe le centre de $\mathfrak{g}$, avec son image dans $\mathfrak{g}\mathfrak{l}(\mathfrak{g})$.
\begin{defi} Une forme lin\'{e}aire $g \in \mathfrak{g}^{\ast}$ est dite de type r\'{e}ductif si son stabilisateur dans $\mathfrak{g}$ pour la repr\'{e}sentation coadjointe, modulo $\mathfrak{z}$, est une alg\`{e}bre de Lie r\'{e}ductive dont le  centre est form\'{e} d'\'{e}l\'{e}ments semi-simples dans $\mathfrak{g}\mathfrak{l}(\mathfrak{g})$.
\end{defi}
De mani\`{e}re \'{e}quivalente, cela revient \`{a} demander que le groupe $G(g)/Z_{\mathfrak{g}}$ soit r\'{e}ductif, o\`{u} $G(g)$ d\'{e}signe  le stabilisateur de $g$  dans $G$ et $Z_{\mathfrak{g}}$  le centralisateur de $\mathfrak{g}$ dans $G$.
\begin{defi} Une alg\`{e}bre de Lie $\mathfrak{g}$ est dite quasi-r\'{e}ductive si elle poss\`{e}de  une forme lin\'{e}aire de type r\'{e}ductif $g \in \mathfrak{g}^{\ast}$.
\end{defi}
La notion de quasi-r\'{e}ductivit\'{e} a \'{e}t\'{e} introduite par Duflo dans \cite{duflo-1982} pour son importance dans la th\'{e}orie des repr\'{e}sentations. Si $\mathfrak{g}$ est réductive, $\mathfrak{g}$ est quasi-r\'{e}ductive puisque $0 \in \mathfrak{g}^{\ast}$ est de type r\'{e}ductif. Il est également vrai que les sous-algèbres de Borel d'une algèbre de Lie réductive sont quasi-réductives (Kostant \cite{kostant2011cascade}, non publi\'{e}, voir \cite{joseph1977}).  A l'exception du type $A$ ou $C$, les sous-algèbres paraboliques d'une algèbre de Lie simple ne sont pas toutes quasi-réductives (voir \cite{Panyushev2005} pour le cas classique et \cite{Anne+Karin} pour le cas exceptionnel). Dans \cite{DKT}, Duflo, Khalgui et Torasso ont caractérisé les sous-alg\`{e}bres paraboliques de $\mathfrak{s}\mathfrak{o}(n,\mathbb{K})$ qui sont quasi-r\'{e}ductives en termes de drapeaux de sous-espaces totalement isotropes stabilisés par ces algèbres. Dans \cite{Anne+Karin}, Baur et Moreau ont donné la liste des sous-algèbres paraboliques des algèbres de Lie de type exceptionnel qui sont quasi-réductives.\\

Tauvel et Yu   ont \'{e}tudi\'{e} dans \cite[ch.40]{Tauvel-Yu} une classe d'alg\`{e}bres de Lie reli\'{e}e \`{a} celle des alg\`{e}bres de Lie quasi-r\'{e}ductives:
\begin{defi} Une forme lin\'{e}aire $g \in \mathfrak{g}^{\ast}$ est dite stable s'il existe un voisinage $V$ de $g$ dans $\mathfrak{g}^{\ast}$ tel que, pour toute forme lin\'{e}aire $f\in V$, les stabilisateurs $\mathfrak{g}(g)$ et $\mathfrak{g}(f)$ soient conjugu\'{e}s par le groupe adjoint alg\'{e}brique de $\mathfrak{g}$.
\end{defi}
\begin{defi}
Une alg\`{e}bre de Lie est dite stable si elle admet une forme lin\'{e}aire stable.
\end{defi}
La notion de stabilit\'{e} a \'{e}t\'{e} introduite par Kosmann et Sternberg dans \cite{Kosmann1974}. D'après \cite{DKT}, toute algèbre de Lie quasi-réductive est stable. Par contre, il existe des alg\`{e}bres de Lie stables qui ne sont pas quasi-r\'{e}ductives (voir \cite[Exemple 2.2.9]{Ammari}).\\

Les sous-algèbres biparaboliques forment une classe intéressante (incluant la classe des sous-algèbres paraboliques et de Levi) d'algèbres de Lie non réductives. Elles sont par définition les intersections de deux sous-algèbres paraboliques de $\mathfrak{g}$ dont la somme est $\mathfrak{g}$.\\

Dans \cite{Panyushev2005}, Panyushev a étudié l'existence des formes linéaires stables pour les sous-algèbres biparaboliques d'une algèbre de Lie simple. Il a établi que lorsque $\mathfrak{q}$ est une sous-algèbre biparabolique dans $\mathfrak{s}\mathfrak{l}_{n}(\mathbb{K})$ ou $\mathfrak{s}\mathfrak{p}_{2n}(\mathbb{K})$, alors $\mathfrak{q}$ admet une forme linéaire stable dont la composante neutre de son stabilisateur est un tore. Il formule alors la conjecture suivante:\\
\\
\textbf{Conjecture de Panyushev} Soit $\mathfrak{q}$ une sous-algèbre biparabolique. Si $\mathfrak{q}$ admet une forme linéaire  stable, alors la composante neutre de son stabilisateur est un tore.\\
Une algèbre de Lie quasi-réductive admet une forme linéaire stable dont le stabilisateur est un tore modulo son centre. On peut alors reformuler la conjecture de Panyushev comme suit:\\
\begin{conj}\label{Conj}
Une sous-algèbre biparabolique admet une forme linéaire stable si et seulement si elle est quasi-réductive.
\end{conj}
Dans \cite{Ammari}, on donne une réponse positive à cette conjecture pour les sous-algèbres paraboliques d'une algèbre de Lie simple orthogonale. Ceci et compte tenu des résultats de Panyushev pour les sous-algèbres paraboliques d'une algèbre de Lie simple de type $A$ ou $C$, donne une réponse positive complète à cette conjecture pour les paraboliques d'une algèbre de Lie  simple classique.\\

Dans  \cite{Ammari2}, nous avons  donné une réponse positive à cette conjecture pour les sous-algèbres paraboliques des algèbres de Lie simples exceptionnelles. Par conséquent, ceci achève la démonstration de cette conjecture pour les paraboliques d'une algèbre de Lie réductive.\\

Le but principal de cet article est d'achever la démonstration de cette conjecture. Compte tenu de ce qui précède et des résultats dans \cite{Panyushev2005} et \cite{Panyushev2017}, il suffit de traiter le cas des biparaboliques qui ne sont pas des paraboliques des algèbres de Lie simples exceptionnelles.  Soit $\mathfrak{q}=\mathfrak{q}_{\pi^{\prime}, \pi^{\prime\prime}}$ une sous-algèbre biparabolique standard d'une algèbre de Lie simple exceptionnelle (voir section \ref{s.alg.bip}). Comme  conséquence du lemme \ref{rang nul}, on ramène l'étude de stabilité des sous-algèbres biparaboliques $\mathfrak{q}$ au cas de rang nul et d'indice strictement positif. Compte tenu du lemme \ref{ty}, il nous suffit d'étudier les sous-algèbres biparaboliques $\mathfrak{q}$ vérifiant la condition $(\ast)$ suivante
$$k_{\pi^{\prime}}+k_{\pi^{\prime\prime}}=\rang(\mathfrak{g})+\dim (E_{\pi^{\prime}}\cap E_{\pi^{\prime\prime}}),\,\, k_{\pi^{\prime}}+k_{\pi^{\prime\prime}}>\rang(\mathfrak{g}) \hbox{ et } \mathcal{K}(\pi^{\prime})\cap \mathcal{K}(\pi^{\prime\prime})=\emptyset.$$
Le théorème \ref{théo pr} montre que  les sous-algèbres biparaboliques   vérifiant la condition $(\ast)$ ne sont pas quasi-réductives si et seulement si elles ne sont pas stables. Sa démonstration utilise des calculs effectués par GAP4 . Notons que GAP4 n'est pas adapté au calcul formel et ne permet pas le calcul du stabilisateur de formes linéaires génériques (voir section \ref{calculs}).

\section{Notations et rappels}
Soient $\mathfrak{g}$ une alg\`{e}bre de Lie, $\mathfrak{g}^{\ast}$ son dual. L'indice de $\mathfrak{g}$, not\'{e} $\mathrm{ind} \mathfrak{g}$, est la dimension minimale des stabilisateurs dans $\mathfrak{g}$ d'un \'{e}l\'{e}ment de $\mathfrak{g}^{\ast}$ pour l'action coadjointe. Il a \'{e}t\'{e} introduit par Dixmier dans \cite{Dixmier} pour son importance dans la th\'{e}orie des repr\'{e}sentations et la th\'{e}orie des orbites.
\begin{defi} Une forme lin\'{e}aire $g\in \mathfrak{g}^{\ast}$ est dite r\'{e}guli\`{e}re si la dimension de son stabilisateur dans $\mathfrak{g}$ pour l'action coadjointe est \'{e}gale \`{a} l'indice de $\mathfrak{g}$.
\end{defi}
\begin{rema}
Il est bien connu que l'ensemble $\mathfrak{g}^{\ast}_{reg}$ des \'{e}l\'{e}ments r\'{e}guliers de $\mathfrak{g}^{\ast}$ est un ouvert de Zariski non vide de $\mathfrak{g}^{\ast}$.
\end{rema}
\begin{defi}
Soit $\mathfrak{g}$ une alg\`{e}bre de Lie alg\'{e}brique. Une forme lin\'{e}aire $g \in \mathfrak{g}^{\ast}$ est dite fortement r\'{e}guli\`{e}re si elle est r\'{e}guli\`{e}re, auquel cas $\mathfrak{g}(g)$
est une alg\`{e}bre de Lie commutative (voir \cite{duflo-vergne-1969}), et si de plus le tore $\mathfrak{j}_{g}$, unique facteur r\'{e}ductif de $\mathfrak{g}(g)$, est de dimension maximale lorsque $g$ parcourt l'ensemble des formes  r\'{e}guli\`{e}res.
\end{defi}
Cette d\'{e}finition est due \`{a} Duflo (voir \cite{duflo-1982}).
\begin{rema}
Il est bien connu que, si $\mathbf{G}$ est un groupe de Lie algébrique d'algèbre de Lie $\mathfrak{g}$,  l'ensemble $\mathfrak{g}^{\ast}_{treg}$ des formes fortement r\'{e}guli\`{e}res  est un ouvert de Zariski $\mathbf{G}$-invariant non vide de $\mathfrak{g}^{\ast}$.\\
Les tores $\mathfrak{j}_{g}$, $g \in \mathfrak{g}^{\ast}$ sont appel\'{e}s les sous-alg\`{e}bres de Cartan-Duflo de $ \mathfrak{g}$. Ils sont deux \`{a} deux conjugu\'{e}s sous l'action du groupe adjoint connexe de $\mathfrak{g}$.
\end{rema}
\begin{defi}
Soit $\mathfrak{g}$ une alg\`{e}bre de Lie alg\'{e}brique sur $\mathbb{K}$. On appelle rang de $\mathfrak{g}$ sur $\mathbb{K}$ et on note $\mathrm{rang} (\mathfrak{g})$ la dimension commune de ses sous-alg\`{e}bres de Cartan-Duflo.
\end{defi}
Dans \cite{Tauvel-Yu2004}, on trouve une caractérisation purement algébrique des formes linéaires stables:
\begin{lem}\label{Tauvel Yu} Soit $f$ un \'{e}l\'{e}ment de $\mathfrak{g}^{\ast}$. Si $\mathfrak{g}$ est une alg\`{e}bre de Lie alg\'{e}brique, on a
$$\hbox{ f est stable si et seulement si } [\mathfrak{g},\mathfrak{g}(f)] \cap \mathfrak{g}(f)=\{0\}.$$
\end{lem}
\begin{rema}
D'après \cite{DKT}, si $\mathfrak{g}$ est quasi-réductive les formes linéaires régulières de type réductif  sont exactement les formes linéaires fortement régulières. De plus elles sont stables.
\end{rema}
Dans cette partie, on renvoie à \cite{Tauvel-Yu} et \cite{Anne+Karin} pour les concepts généraux utilisés.\\

Soient $\mathfrak{g}$ une algèbre de Lie semi-simple sur $\mathbb{K}$, $G$ son groupe adjoint, $K$ sa forme de Killing, $\mathfrak{h}$ une sous-algèbre de Cartan de $\mathfrak{g}$, et $\Delta$ le système de racines du couple $(\mathfrak{g},\mathfrak{h})$. On fixe une base $\pi$ de $\Delta$ (que l'on notera $\Delta_{\pi}$), et on désigne par $\Delta_{\pi}^{+}$ (resp. $\Delta_{\pi}^{-}$) l'ensemble des racines positives (resp. négatives) associé. Pour toute partie $\pi^{\prime}$ de $\pi$, on note $\mathbb{Z}\pi^{\prime}$ (resp. $\mathbb{N}\pi^{\prime}$) l'ensemble des combinaisons linéaires à coefficients dans $\mathbb{Z}$ (resp. dans $\mathbb{N}$) des éléments de $\pi^{\prime}$. On pose $$\Delta_{\pi^{\prime}}=\Delta_{\pi}\cap \mathbb{Z}\pi^{\prime},\,\,\Delta_{\pi^{\prime}}^{+}=\Delta_{\pi}\cap \mathbb{N}\pi^{\prime}=\Delta_{\pi}^{+}\cap \Delta_{\pi^{\prime}}.$$
Ainsi $\Delta_{\pi^{\prime}}$ est un système de racines dont $\pi^{\prime}$ est une base de racines simples et $\Delta_{\pi^{\prime}}^{+}$ est l'ensemble des racines positives correspondant.\\

Si $\alpha \in \Delta_{\pi} $, $\mathfrak{g}^{\alpha}$ est le sous-espace radiciel associé à $\alpha$. On pose $$\mathfrak{n}^{\pm}=\sum_{\alpha\in \Delta_{\pi}^{\pm}} \mathfrak{g}^{\alpha},\,\, \mathfrak{b}^{\pm}=\mathfrak{h}\oplus\mathfrak{n}^{\pm}.$$
On désigne par $h_{\alpha}$ l'unique élément de $[\mathfrak{g}^{\alpha}, \mathfrak{g}^{-\alpha}]$ tel que $\alpha(h_{\alpha})=2$. Si $\lambda \in \mathfrak{h}^{\ast}$, on écrit $\langle \lambda,\alpha^{\vee}\rangle$ pour $\lambda(h_{\alpha})$.\\

Soit $S$ une partie de $\pi$. Par récurrence sur le cardinal de $S$, on définit un sous-ensemble $\mathcal{K}(S)$ de l'ensemble des parties de $S$ de
la manière suivante (la notion de connexité est relative au diagramme de Dynkin):\\
(i) $\mathcal{K}(\emptyset)=\emptyset$.\\
(ii) Si $S_{1},S_{2}...,S_{r}$ sont les composantes connexes de $S$, on a:
$$\mathcal{K}(S)=\mathcal{K}(S_{1})\cup...\cup\mathcal{K}(S_{r})$$
(iii) Si $S$ est connexe, alors: $$\mathcal{K}(S)=\{S\}\cup \mathcal{K}(\{\alpha \in S; \langle \alpha,\varepsilon_{S}^{\vee}\rangle=0  \}),$$
  $\varepsilon_{S}$ étant la plus grande racine de $\Delta_{S}$.\\

Cette construction est due à  Kostant (voir  \cite{kostant2011cascade} et \cite{kostant2012center}). On a les résultats suivants:
\begin{lem}(voir  \cite[ch.40]{Tauvel-Yu})\\
$\text{i})$ Tout $K \in \mathcal{K}(S)$ est une partie connexe de $\pi$.\\
$\text{ii})$ Si $K, K^{\prime} \in \mathcal{K}(S)$, alors ou $K \subset K^{\prime}$, ou $K^{\prime}\subset K$, ou $K$ et $ K^{\prime}$ sont des parties disjointes de $S$ telles que $\alpha + \beta \notin \Delta_{\pi}$ si $\alpha \in \Delta_{K}$ et $\beta \in \Delta_{K^{\prime}}$.\\
$\text{iii})$  $\{\varepsilon_{K}; K \in \mathcal{K}(S)\}$ est un ensemble de racines deux à deux fortement orthogonales de $\Delta_{\pi}$.
\end{lem}
On note  $k_{\pi^{\prime}}$ le cardinal de $\mathcal{K}(\pi^{\prime})$ pour tout $\pi^{\prime}$ inclus dans $\pi$ et on remarque que $k_{\pi}$ ne dépend que de $\mathfrak{g}$. On donne dans le tableau ci-dessous la valeur de $k_{\pi}$ pour les différents types d'algèbres de Lie simples. Si $r$ est un nombre rationnel, $[r]$ désigne sa partie entière.
\\
\begin{table}[h]
\begin{center}
\begin{tabular}{|c|c|c|c|c|c|c|c|c|c|}
\hline
& & & & & & & & & \\
$k_{\pi}$&$A_{l}$ &$B_{l}$& $C_{l}$&$D_{l}$& $E_{6}$ & $E_{7}$ & $E_{8}$ & $F_{4}$ & $G_{2}$\\
& & & & & & & & &\\
\hline
& & & & & & & & & \\
 &$[\frac{l+1}{2}]$ & $l$ & $l$ & $2[\frac{l}{2}]$ & 4 & 7 & 8 & 4 & 2 \\
& & & & & & & & & \\
\hline
\end{tabular}
\vspace{.5cm}
\caption{\label{Tkg} $k_{\pi}$ pour les différents  types d'algèbres de Lie simples.}
\end{center}
\end{table}
\\
\newpage
Pour $\pi^{\prime}$ une partie de $\pi$, on note  $\mathcal{E}_{\pi^{\prime}}$ l'ensemble des racines $\varepsilon_{K}$, avec $ K \in \mathcal{K}(\pi^{\prime})$. On décrit ci-dessous $\mathcal{E}_{\pi}$ dans le cas des algèbres de Lie simples exceptionnelles (dans les cas de type $F_{4}$, $E_{6}$, $E_{7}$ et $E_{8}$, on utilise les notations de Bourbaki \cite[pages 272, 260, 264 et 268]{Bourbaki} pour l'écriture des racines dans une base $\pi$):
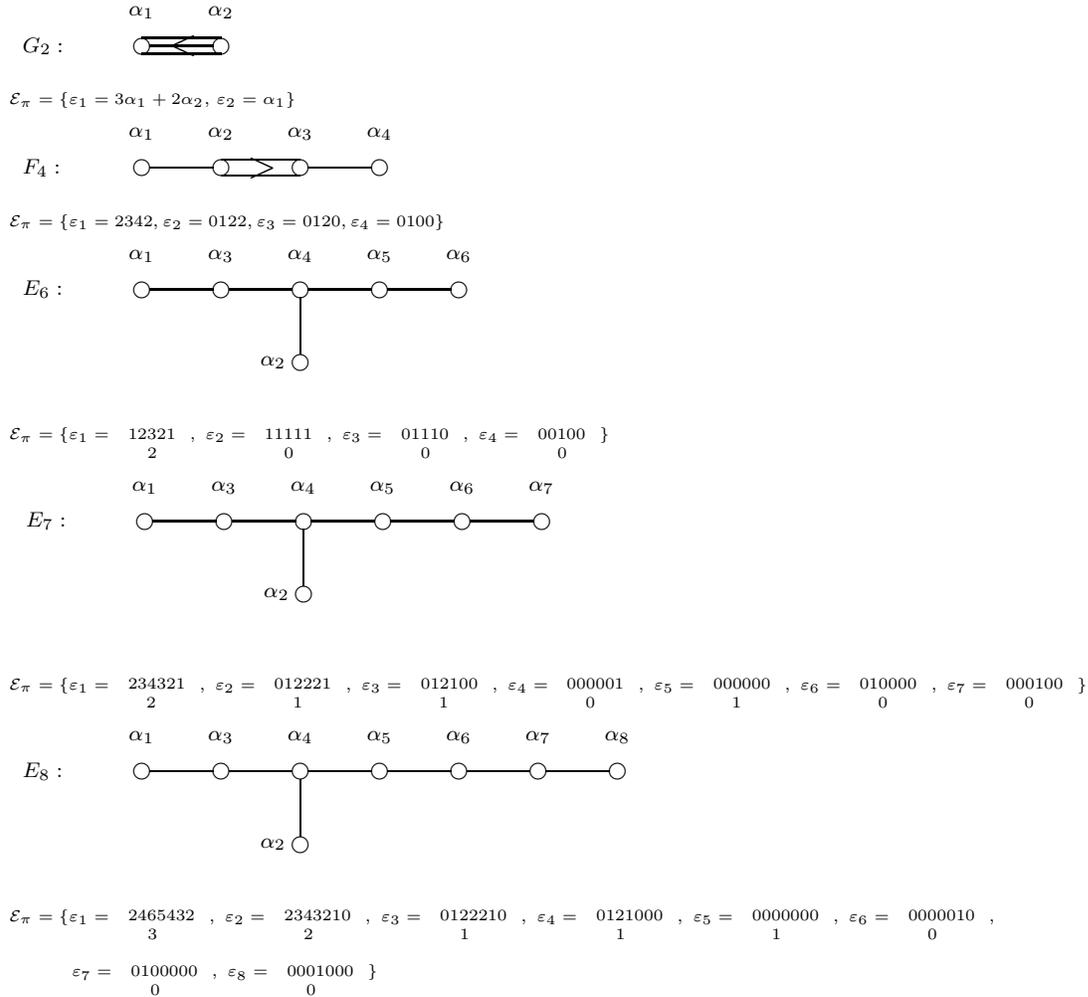
$${\tiny
\begin{array}{l}
 \begin{Dynkin}
   \Dbloc{\Dtext{l}{G_2:}}
   \Dbloc{ \Dcirc \Ddoubleeast \Deast \Dtext{t}{\alpha_{1}}}
   \Dleftarrow
   \Dbloc{\Dwest \Ddoublewest \Dcirc \Dtext{t}{\alpha_{2}}}
  \end{Dynkin}\\
 \mathcal{E}_{\pi}= \{\varepsilon_{1} = 3\alpha_{1}+2\alpha_{2}, \, \varepsilon_{2} = \alpha_{1}\} \\
  \\
  \begin{Dynkin}
   \Dbloc{\Dtext{l}{F_4:}}
   \Dbloc{\Dcirc \Deast \Dtext{t}{\alpha_{1}}}
   \Dbloc{\Dwest \Dcirc \Ddoubleeast  \Dtext{t}{\alpha_{2}}}
   \Drightarrow
   \Dbloc{\Ddoublewest \Dcirc \Deast \Dtext{t}{\alpha_{3}}}
   \Dbloc{\Dwest \Dcirc \Dtext{t}{\alpha_{4}}}
  \end{Dynkin} \\
 \mathcal{E}_{\pi}=\{\varepsilon_{1} = 2342,  \varepsilon_{2} = 0122,  \varepsilon_{3}=0120,  \varepsilon_{4}=0100 \} \\
 \\
\begin{Dynkin}
\Dbloc{\Dtext{l}{E_6:}}
\Dbloc{\Dcirc \Deast\Dtext{t}{\alpha_{1}} }
 \Dbloc{\Dwest\Dcirc \Deast\Dtext{t}{\alpha_{3}} }
 \Dbloc{\Dwest \Dcirc \Dsouth \Deast\Dtext{t}{\alpha_{4}} }
 \Dbloc{\Dwest \Dcirc \Deast\Dtext{t}{\alpha_{5}}}
\Dbloc{ \Dwest \Dcirc\Dtext{t}{\alpha_{6}} }
\Dskip
\Dspace \Dspace \Dspace \Dbloc{\Dnorth \Dcirc\Dtext{l}{\alpha_{2}}}
\end{Dynkin} \\
\mathcal{E}_{\pi}=\{\varepsilon_{1}=\begin{array}{c}\\12321\\2\end{array}, \
\varepsilon_{2}=\begin{array}{c}\\11111\\0\end{array}, \
\varepsilon_{3}=\begin{array}{c}\\01110\\0\end{array}, \
\varepsilon_{4}=\begin{array}{c}\\00100\\0\end{array} \} \\
\\
\,
\begin{Dynkin}
\Dbloc{\Dtext{l}{E_7:}}
\Dbloc{\Dcirc \Deast\Dtext{t}{\alpha_{1}} }
 \Dbloc{\Dwest\Dcirc \Deast\Dtext{t}{\alpha_{3}} }
 \Dbloc{\Dwest \Dcirc \Dsouth \Deast\Dtext{t}{\alpha_{4}} }
 \Dbloc{\Dwest \Dcirc \Deast\Dtext{t}{\alpha_{5}}}
\Dbloc{ \Dwest  \Deast \Dcirc\Dtext{t}{\alpha_{6}} }
\Dbloc{ \Dwest \Dcirc\Dtext{t}{\alpha_{7}} }
\Dskip
\Dspace \Dspace \Dspace \Dbloc{\Dnorth \Dcirc\Dtext{l}{\alpha_{2}}}
\end{Dynkin}\\
\\
\mathcal{E}_{\pi}= \{ \varepsilon_{1}=\begin{array}{c}\\234321\\ \hspace{-0.15cm}2\end{array}, \
\varepsilon_{2}=\begin{array}{c}\\012221\\ \hspace{-0.15cm}1\end{array}, \
\varepsilon_{3}=\begin{array}{c}\\012100\\ \hspace{-0.15cm}1\end{array}, \
\varepsilon_{4}=\begin{array}{c}\\000001\\ \hspace{-0.15cm}0\end{array}, \
\varepsilon_{5}=\begin{array}{c}\\000000\\ \hspace{-0.15cm}1\end{array}, \
\varepsilon_{6}=\begin{array}{c}\\010000\\ \hspace{-0.15cm}0\end{array}, \
\varepsilon_{7}=\begin{array}{c}\\000100\\ \hspace{-0.15cm}0\end{array}\}\\
\\
\begin{Dynkin}
\Dbloc{\Dtext{l}{E_8:}}
\Dbloc{\Dcirc \Deast\Dtext{t}{\alpha_{1}} }
 \Dbloc{\Dwest\Dcirc \Deast\Dtext{t}{\alpha_{3}} }
 \Dbloc{\Dwest \Dcirc \Dsouth \Deast\Dtext{t}{\alpha_{4}} }
 \Dbloc{\Dwest \Dcirc \Deast\Dtext{t}{\alpha_{5}}}
\Dbloc{ \Dwest  \Deast \Dcirc\Dtext{t}{\alpha_{6}} }
\Dbloc{ \Dwest \Deast \Dcirc\Dtext{t}{\alpha_{7}} }
\Dbloc{ \Dwest \Dcirc\Dtext{t}{\alpha_{8}} }
\Dskip
\Dspace \Dspace \Dspace \Dbloc{\Dnorth \Dcirc\Dtext{l}{\alpha_{2}}}
\end{Dynkin} \\
\mathcal{E}_{\pi}=\{ \varepsilon_{1}=\begin{array}{c}\\2465432\\ \hspace{-0.25cm}3\end{array}, \
\varepsilon_{2}=\begin{array}{c}\\2343210\\ \hspace{-0.25cm}2\end{array}, \
\varepsilon_{3}=\begin{array}{c}\\0122210\\ \hspace{-0.25cm}1\end{array}, \
\varepsilon_{4}=\begin{array}{c}\\0121000\\ \hspace{-0.25cm}1\end{array}, \
\varepsilon_5=\begin{array}{c}\\0000000\\ \hspace{-0.25cm}1\end{array}, \
\varepsilon_{6}=\begin{array}{c}\\0000010\\ \hspace{-0.25cm}0\end{array},\\
\hspace{0.83cm} \varepsilon_{7}=\begin{array}{c}\\0100000\\ \hspace{-0.25cm}0\end{array}, \
\varepsilon_{8}=\begin{array}{c}\\0001000\\ \hspace{-0.25cm}0\end{array} \}\\
\end{array}
}
$$
\section{Sous-algèbres biparaboliques}\label{s.alg.bip}
On garde les notations précédentes et on désigne par $\mathfrak{p}_{\pi^{\prime}}^{+}$ la sous-algèbre de  $\mathfrak{g}$ engendrée par $ \mathfrak{b}^{+}$ et les $\mathfrak{g}^{-\alpha}$, $\alpha \in \pi^{\prime}$: c'est une sous-algèbre parabolique de $\mathfrak{g}$ et toutes les sous-algèbres paraboliques de $\mathfrak{g}$ sont à conjugaison près obtenues ainsi. On note $\mathfrak{p}_{\pi^{\prime}}^{-}$ la sous-algèbre parabolique opposée à $\mathfrak{p}_{\pi^{\prime}}^{+}$, i.e, la sous-algèbre engendrée par $\mathfrak{b}^{-}=\mathfrak{h}\oplus\mathfrak{n}^{-}$ et les $\mathfrak{g}^{\alpha}$, $\alpha \in \pi^{\prime}$.\\

Soient $\pi^{\prime}$, $\pi^{\prime\prime}$ deux parties de $\pi$. On note $\mathfrak{q}_{\pi^{\prime}, \pi^{\prime\prime}}$ l'intersection de deux sous-algèbres paraboliques $\mathfrak{p}_{\pi^{\prime}}^{+}$ et $\mathfrak{p}_{\pi^{\prime\prime}}^{-}$. La sous-algèbre $\mathfrak{q}_{\pi^{\prime}, \pi^{\prime\prime}}$ est dite sous-algèbre biparabolique associée à  $\pi^{\prime}$ et $\pi^{\prime\prime}$. Il est bien connu que toute sous-algèbre biparabolique est conjuguée à une des $\mathfrak{q}_{\pi^{\prime}, \pi^{\prime\prime}}$.\\

Si $v \in \mathfrak{g}$ et $\mathfrak{a}$ est une sous-algèbre de $\mathfrak{g}$, on note $\varphi_{v}$ la forme linéaire
sur $\mathfrak{a}$ définie par $\varphi_{v}(x)= K(v,x)$ pour tout $x \in \mathfrak{a}$. Si $\mathfrak{a}=\mathfrak{q}_{\pi^{\prime}, \pi^{\prime\prime}}$, l'application $v\mapsto \varphi_{v}$ induit un isomorphisme d'espaces vectoriels de  $\mathfrak{q}_{\pi^{\prime\prime}, \pi^{\prime}}$ sur $(\mathfrak{q}_{\pi^{\prime}, \pi^{\prime\prime}})^{\ast}$.\\
\section{Réduction au cas de rang nul}
On note $E_{\pi^{\prime},\pi^{\prime\prime}}$ (resp. $E_{\pi^{\prime}}$) le sous-espace de $\mathfrak{h}^{\ast}$ engendré par $\mathcal{E}_{\pi^{\prime}}\cup \mathcal{E}_{\pi^{\prime\prime}}$ (resp. $\mathcal{E}_{\pi^{\prime}}$). Ainsi,
  $\dim E_{\pi^{\prime},\pi^{\prime\prime}}= k_{\pi^{\prime}}+k_{\pi^{\prime\prime}}-\dim (E_{\pi^{\prime}}\cap E_{\pi^{\prime\prime}})$.\\

Soit $f \in \mathfrak{q}_{\pi^{\prime}, \pi^{\prime\prime}}^{\ast}$ définie par $$f=\sum_{K \in \mathcal{K}(\pi^{\prime\prime})} a_{K} X^{\ast}_{\varepsilon_{K}} + \sum_{L \in \mathcal{K}(\pi^{\prime})} b_{L} X^{\ast}_{-\varepsilon_{L}},$$
o\`{u} les $a_{K}$ et les $b_{L}$ sont des scalaires non nuls.\\

Il existe un ouvert de Zariski dans $\mathbb{C}^{\mathcal{K}(\pi^{\prime \prime})} \times \mathbb{C}^{\mathcal{K}(\pi^{\prime})}$ tel que pour $((a_{K})_{K \in \mathcal{K}(\pi^{\prime\prime})},(b_{L})_{L \in \mathcal{K}(\pi^{\prime})})$ dans cet ouvert  $f$ soit régulière. D'après \cite{MR2264140}, on a
\begin{equation}\ind(\mathfrak{q}_{\pi^{\prime}, \pi^{\prime\prime}})=(\rang(\mathfrak{g})- \dim E_{\pi^{\prime},\pi^{\prime\prime}})+(k_{\pi^{\prime}}+k_{\pi^{\prime\prime}}-\dim E_{\pi^{\prime},\pi^{\prime\prime}}).
\end{equation}
\begin{rema}
On a $$\ind(\mathfrak{q}_{\pi^{\prime}, \pi^{\prime\prime}})={0} \hbox{ si et seulement si }\,\,  E_{\pi^{\prime}}\cap E_{\pi^{\prime\prime}}=\{0\} \hbox{ et }\,\, k_{\pi^{\prime}}+k_{\pi^{\prime\prime}}=\rang(\mathfrak{g}).$$
\end{rema}
\begin{lem}\label{ty}(voir \cite{Tauvel-Yu2005})
Soit $f$ comme précédemment. Alors,  $\mathfrak{q}_{\pi^{\prime}, \pi^{\prime\prime}}(f)$ contient une sous-algèbre commutative, formée d'éléments semi-simples, et de dimension: $$\rang(\mathfrak{g})- \dim E_{\pi^{\prime},\pi^{\prime\prime}}+card(\mathcal{K}(\pi^{\prime})\cap \mathcal{K}(\pi^{\prime\prime})).$$
\end{lem}

La proposition et le lemme qui suivent permettent de ramener  l'étude de la stabilité des sous-algèbres biparaboliques au cas de rang nul.\\
\begin{pr}
Soient $\mathfrak{g}$ une algèbre de Lie simple, $\mathfrak{q}$ une sous-algèbre biparabolique de rang non nul de $\mathfrak{g}$, $\mathfrak{j}$ une sous-algèbre de Cartan-Duflo de $\mathfrak{q}$ et $\mathfrak{h}$ une sous-algèbre de Cartan de $\mathfrak{g}$ contenant $\mathfrak{j}$ et contenues dans $\mathfrak{q}$. Alors, on a
$$\mathfrak{q}^{\mathfrak{j}}=\mathfrak{z}_{\Delta^{\mathfrak{j}}}\oplus\mathfrak{q}^{\mathfrak{j},1},$$
 o\`{u} $\mathfrak{z}_{\Delta^{\mathfrak{j}}}$ est le centre de $\mathfrak{q}^{\mathfrak{j}}$  et $\mathfrak{q}^{\mathfrak{j},1}$ est un produit direct de sous-algèbres biparaboliques de rang nul.
\end{pr}
 \begin{proof}
 Soient $\mathfrak{g}$ une algèbre de Lie simple, $\mathfrak{q}$ une sous-algèbre biparabolique de rang non nul de $\mathfrak{g}$, $\mathfrak{j}$ une sous-algèbre de Cartan-Duflo de $\mathfrak{q}$ et $\mathfrak{h}$ une sous-algèbre de Cartan de $\mathfrak{g}$ contenant $\mathfrak{j}$ et contenues dans $\mathfrak{q}$. On a $\dim\mathfrak{j}>0$ et $\mathfrak{j}\subset \mathfrak{h}\subset\mathfrak{q}^{\mathfrak{j}}\subset \mathfrak{g}^{\mathfrak{j}}$ et
$\mathfrak{g}^{\mathfrak{j}}=\mathfrak{h}\oplus (\bigoplus_{\alpha \in \Delta, \alpha|_{\mathfrak{j}}=0}\mathfrak{g}_{\alpha})$.\\
Soit $\mathfrak{z}_{\Delta^{\mathfrak{j}}}=\{\mathfrak{h}\cap(\cap_{\alpha|_{\mathfrak{j}}=0} \ker\alpha)\}$ qui n'est autre que le centre de
$\mathfrak{g}^{\mathfrak{j}}$. Ainsi, on a $$\mathfrak{g}^{\mathfrak{j}}=\mathfrak{z}_{\Delta^{\mathfrak{j}}}\oplus \mathfrak{g}_{1}\oplus
\mathfrak{g}_{2}\oplus...\oplus\mathfrak{g}_{k},$$
o\`{u} les $\mathfrak{g}_{i}$, pour $1\leq i \leq k$, sont des algèbres de Lie simples puisque l'algèbre de Lie $\mathfrak{g}^{\mathfrak{j}}$ est réductive. Comme chacune des sous-algèbres de Lie $\mathfrak{q}^{\mathfrak{j},1}$ et $\mathfrak{g}_{i}$, $1\leq i \leq k$ est somme directe de son intersection avec $\mathfrak{h}$ et des sous-espaces radiciels sous l'action de $\mathfrak{h}$ qu'elle contient, on a
$\mathfrak{q}^{\mathfrak{j}}=\mathfrak{z}_{\Delta^{\mathfrak{j}}}\oplus\mathfrak{q}^{\mathfrak{j},1}$ avec $\mathfrak{q}^{\mathfrak{j},1}$ produit direct des sous-algèbres biparaboliques $\mathfrak{q}_{l}^{\mathfrak{j}}=\mathfrak{q}^{\mathfrak{j}}\cap \mathfrak{g}_{l}$ des $\mathfrak{g}_{l}$, $1\leq l \leq k$ qui sont de rang nul.
\end{proof}
\begin{lem}\label{rang nul}(voir\cite{Ammari})
Soit $\mathfrak{g}$ une algèbre de Lie algébrique, $\mathfrak{j}$ une sous-algèbre de Cartan-Duflo de $\mathfrak{g}$ et   $\mathfrak{g}^{\mathfrak{j}}$ le centralisateur de $\mathfrak{j}$ dans $\mathfrak{g}$. Alors
$$\mathfrak{g} \hbox{ est stable si et seulement si } \mathfrak{g}^{\mathfrak{j}} \hbox{ est stable  }.$$
\end{lem}
\begin{rema}
 D'après ce qui précède et pour donner une réponse positive à la conjecture \ref{Conj}, il suffit de traiter le cas des sous-algèbres biparaboliques  qui sont de rang nul. Pour ce qui concerne les sous-algèbres biparaboliques des algèbres de Lie simples classiques de type $A$, $B$ ou $C$, le problème se ramène au cas des sous-algèbres paraboliques (voir \cite{Panyushev2005}). Dans le cas de type $D$, il y a deux types de sous-algèbres biparaboliques. D'une part celles qui stabilisent une paire de drapeaux opposés de sous-espaces totalement isotropes, pour lesquelles le procédé de réduction de Panyushev ramène encore le problème au cas des sous-algèbres paraboliques. D'autre part celles  qui stabilisent une paire de drapeaux croisés de sous-espaces totalement isotropes, lesquelles  sont toutes quasi-réductives (voir \cite{Panyushev2017} ou \cite{Djebali}). Enfin dans \cite{Ammari} et \cite{Ammari2}, on trouve une réponse positive à cette conjecture pour le cas des sous-algèbres paraboliques d'une algèbre de Lie simple. Par conséquent, il suffit de traiter le cas des sous-algèbres biparaboliques d'une algèbre de Lie simple exceptionnelle qui ne sont pas des paraboliques et qui sont de rang nul pour donner une réponse complète à la conjecture \ref{Conj}.\\

\end{rema}
\section{Stabilité des sous-algèbres biparaboliques $\mathfrak{q}_{\pi^{\prime}, \pi^{\prime\prime}}$}
Soit $\mathfrak{g}$ une algèbre de Lie simple de type exceptionel, $\mathfrak{h}\subset \mathfrak{g}$ une sous-algèbre de Cartan et $\pi$ une base de racines simples du système de racines de $\mathfrak{h}$ dans $\mathfrak{g}$. Soit $\mathfrak{q}=\mathfrak{q}_{\pi^{\prime}, \pi^{\prime\prime}}$ une sous-algèbre biparabolique standard de $\mathfrak{g}$, o\`{u} $\pi^{\prime}$ et $\pi^{\prime\prime}$ sont des parties de $\pi$.\\
Compte tenu du lemme \ref{ty}, si $\mathfrak{q}$ est de rang nul, on a $$k_{\pi^{\prime}}+k_{\pi^{\prime\prime}}=\rang(\mathfrak{g})+\dim (E_{\pi^{\prime}}\cap E_{\pi^{\prime\prime}}) \hbox{ et } \mathcal{K}(\pi^{\prime})\cap \mathcal{K}(\pi^{\prime\prime})=\emptyset$$
au quel cas $$\ind(\mathfrak{q})=k_{\pi^{\prime}}+k_{\pi^{\prime\prime}}-\rang(\mathfrak{g}).$$
Par suite, il nous suffit d'étudier les sous-algèbres biparaboliques  vérifiant la condition $(\ast)$ suivante $$k_{\pi^{\prime}}+k_{\pi^{\prime\prime}}=\rang(\mathfrak{g})+\dim (E_{\pi^{\prime}}\cap E_{\pi^{\prime\prime}}),\,\, k_{\pi^{\prime}}+k_{\pi^{\prime\prime}}>\rang(\mathfrak{g}) \hbox{ et } \mathcal{K}(\pi^{\prime})\cap \mathcal{K}(\pi^{\prime\prime})=\emptyset.$$
Désormais, on suppose que $\mathfrak{q}$ vérifie la condition $(\ast)$ et on désigne par $A$ le sous-espace vectoriel de $\mathfrak{q}^{\ast}$ engendré par $\{X^{\ast}_{\xi}, \xi \in \mathcal{E}_{\pi^{\prime\prime}}\}\cup \{X^{\ast}_{-\xi}, \xi \in \mathcal{E}_{\pi^{\prime}}\}$. Alors $$\dim(A)=k_{\pi^{\prime}}+k_{\pi^{\prime\prime}}\geq \ind(\mathfrak{q}).$$
\begin{pr}
Il existe $f \in A$ tel que $A+\mathfrak{q}.f=\mathfrak{q}^{\ast}$.
\end{pr}
\begin{proof}
On choisit $f \in A$ et on utilise GAP pour montrer  le résultat pour cet élément.
\end{proof}

\begin{co}
Les ensembles  $\mathfrak{q}^{\ast}_{reg}\cap A$ et $\mathfrak{q}^{\ast}_{treg}\cap A$ sont des ouverts de Zariski non vides de A.
\end{co}
On veut montrer que l'algèbre de Lie $\mathfrak{q}$ n'est pas stable, si elle n'est pas quasi-réductive. On sait qu'il suffit de le faire lorsque $\mathfrak{q}$ est de rang nul et d'indice non nul, donc lorsqu'elle vérifie la condition $(\ast)$. Cependant, il y a des algèbres biparaboliques qui vérifient la condition $(\ast)$ et qui ne sont pas de rang nul: il faut donc les éliminer de la liste de celles à traiter. Pour ce faire on procède ainsi: on choisit $f \in A$ au hasard, puis à l'aide de GAP4 on calcule son stabilisateur $\mathfrak{q}(f)$ et on vérifie qu'elle est régulière, c'est à dire que $\dim(\mathfrak{q}(f))= \ind(\mathfrak{q})$. On calcule ensuite l'orthogonal $\mathfrak{q}(f)^{\perp_{K}}$ de $\mathfrak{q}(f)$ pour la forme de Killing $K$ de $\mathfrak{g}$, puis son intersection $\mathfrak{q}(f)^{\perp_{K}}\cap \mathfrak{q}(f)$ avec $\mathfrak{q}(f)$. Si ce dernier espace est nul, la forme linéaire $f$ est de type réductif et donc le rang de $\mathfrak{q}$ est strictement positif et on peut l'éliminer. Notons que parmi les sous-algèbres biparaboliques qui restent,  il y en a   qui ne sont pas de rang nul (voir la section $\ref{GAP})$.\\

Après avoir appliqué ce procédé, on constate   que les algèbres restantes sont toutes d'indice 1, 2 ou 3. On désigne $\mathcal{L}$ leur liste.\\

Par ailleurs, on a le résultat suivant
\begin{pr}
Pour la sous-algèbre biparabolique $\mathfrak{q}$ vérifiant la condition $(\ast)$, les assertions suivantes sont équivalentes:\\
(i) $\mathfrak{q}$ n'est pas stable.\\
(ii) Il existe  un ouvert de Zariski  $A^{\prime}$ de $A$  tel que, pour tout $f \in A^{\prime}$, $[\mathfrak{q},\mathfrak{q}(f)]\cap \mathfrak{q}(f)\neq \{0\}$.
\end{pr}
Compte tenu de ce qui précède, pour montrer qu'une sous-algèbre biparabolique d'une algèbre de Lie simple de type exceptionnel  n'est pas stable, il suffit de montrer qu'il existe $h\in \mathfrak{h}$ et $A^{\prime}$  ouvert de Zariski de $A$ tel que
$$\hbox{ pour tout } f \in A^{\prime}, \hbox{ il existe } X \in \mathfrak{q}(f)  \hbox{ différent de 0 tel que } h.f=-f \hbox{ et } h.X=-X$$
Ce résultat est établi dans la section $\ref{GAP}$ pour les sous-algèbres biparaboliques de la liste $\mathcal{L}$.\\

Le résultat suivant est le but principal de ce travail
\begin{theo}\label{théo pr}
Soit $\mathfrak{q}=\mathfrak{q}_{\pi^{\prime}, \pi^{\prime\prime}}$ une sous-algèbre biparabolique d'une algèbre de Lie simple exceptionnelle qui vérifie la condition $(\ast)$. Alors les deux conditions suivantes sont équivalentes:\\
(i) $\rang(\mathfrak{q})=0$ et $\ind(\mathfrak{q})>0$.\\
(ii) Il existe  $A^{\prime}\subset A$ un ouvert de Zariski  et $h \in \mathfrak{h}$ tels que  $ \forall f \in A^{\prime}$,  $h.f=-f$ et $[h,x]=-x$, $\forall x \in \mathfrak{q}(f)$.
\end{theo}
\begin{proof}
$(ii)\Rightarrow (i)$ Supposons qu'il existe un élément $h\in \mathfrak{h}$ et un ouvert de Zariski $A^{\prime}$ de $A$ tels que $h.f=-f$, $\forall f \in A^{\prime}$ et $[h,x]=-x$, $\forall x \in \mathfrak{q}(f)$. Supposons de plus qu'il existe un élément $X\in \mathfrak{q}(f)$ semi-simple. Soient $\lambda$ une valeur propre de $adX$, $V_{\lambda}$ le sous-espace propre correspondant de  $adX$ dans $\mathfrak{q}$ et  $v \in V_{\lambda}$. On a $$h.X.v=[h,X].v+ X.h.v=-X.v+\lambda X.v=(\lambda-1) X.v $$
ainsi et compte tenu de ce qui précède, l'endomorphisme $X$ envoie le sous-espace $V_{\lambda}$ dans $V_{\lambda-1}$. Par suite, l'endomorphisme $X$ est nilpotent. Ceci contredit le fait que $X$ est semi-simple et donc le rang de $\mathfrak{q}$ est nul. D'autre part et comme $\mathfrak{q}$ est non stable, $\mathfrak{q}$ est non quasi-réductive. Ceci implique que l'indice de $\mathfrak{q}$ est strictement positif.\\
$(i)\Rightarrow (ii)$ Soit $\mathfrak{q}=\mathfrak{q}_{\pi^{\prime}, \pi^{\prime\prime}}$ une sous-algèbre biparabolique d'une algèbre de Lie simple exceptionnelle qui vérifie la condition $(\ast)$.  Par construction et en utilisant GAP4 (voir ci-après), on vérifie que si $\mathfrak{q}$ est de rang nul et d'indice strictement positif il existe $h \in \mathfrak{h}$  et $A^{\prime}$ ouvert de Zariski dans $A$ tel que $h.f=-f$, $\forall f \in A^{\prime}$ et $[h,x]=-x$, $\forall x \in \mathfrak{q}(f)$.
\end{proof}
\begin{rema}
 Parmi les sous-algèbres $\mathfrak{q}_{\pi^{\prime}, \pi^{\prime\prime}}$ qui vérifient la condition $(\ast)$, il en existe qui:\\
-  sont quasi-réductives.\\
- qui sont de rang nul et d'indice strictement positif.\\
- qui sont de rang non nul et non quasi-réductives.
\end{rema}

\section{Présentation des calculs avec GAP4}\label{GAP}
Soit $\mathfrak{g}$ une algèbre de Lie simple exceptionnelle et $\mathfrak{q}=\mathfrak{q}_{\pi^{\prime}, \pi^{\prime\prime}}$ une sous-algèbre biparabolique de $\mathfrak{g}$ vérifiant la condition $(\ast)$. Grâce à GAP4, étant donnée une forme linéaire $f \in \mathfrak{q}^{\ast}$, on peut effectuer le calcul du stabilisateur $\mathfrak{q}(f)$  afin de déterminer si $f$ est régulière. Dans ce cas, le calcul à l'aide de GAP4  de  l'intersection  de $\mathfrak{q}(f)$ avec son orthogonal pour la forme de Killing  permet de déterminer la dimension de son facteur réductif.\\

La démarche générale est la suivante: on définit l'algèbre de Lie $\textbf{L}$ dans laquelle on veut travailler grâce à la commande "SimpleLieAlgebra", on définit un système de racines "RootSystem", un système de racines positives correspondant "PositiveRoots"  puis des systèmes de vecteurs positifs et négatifs
associés "PositiveRootVectors" et  "NegativeRootVectors". La commande "CanonicalGenerators" donne une base de la sous-algèbre de Cartan. On peut désormais
faire des calculs dans l'algèbre de Lie $\textbf{L}$.\\

Pour définir la sous-algèbre biparabolique $P:=\mathfrak{q}_{\pi^{\prime}, \pi^{\prime\prime}}$ et ses générateurs, on utilise les commandes "Concatenation" et "Subalgebra". La commande "List(Basis(P))" donne une base de la sous-algèbre biparabolique $P$, que l'on note $bP$. Ensuite, on définit l'élément $v=\sum_{K \in  \mathcal{K}(\pi^{\prime})} a_{K} X_{\varepsilon_{K}}+\sum_{L \in  \mathcal{K}(\pi^{\prime\prime})} b_{L} X_{-\varepsilon_{L}}$ de $\mathfrak{q}_{\pi^{\prime\prime}, \pi^{\prime}}$ et on note $f \in P^{\ast}$ la forme linéaire qui lui est associée ($f=\sum_{K \in \mathcal{K}(\pi^{\prime\prime})} a_{K} \alpha(K) X^{\ast}_{\varepsilon_{K}} + \sum_{L \in \mathcal{K}(\pi^{\prime})} b_{L} \alpha(L) X^{\ast}_{-\varepsilon_{L}}$, o\`{u} pour toute partie connexe $K$ de $\pi$, $\alpha(K)=Kappa(X_{\varepsilon_{K}},X_{-\varepsilon_{K}})$, Kappa  désignant la forme de Killing sur $\textbf{L}$). Notons ici que GAP4 n'est pas adapté au calcul formel (il faut donner donc des valeurs pour les $a_{K}$ et les $b_{L}$). Les étapes suivantes permettent de calculer le stabilisateur de $f$ dans $P$: On calcule tout d'abord l'espace $V$ engendré par les crochets $v \ast bP[i]$, pour $i=1,...,dp$, $dp$ désigne la dimension de la sous-algèbre biparabolique  $P$. On obtient l'orthogonal $K$ de l'espace $V$ pour la forme de Killing avec la commande "KappaPerp". Ainsi, le stabilisateur de $f$ dans  $P$, que l'on note $S$, n'est autre que l'intersection de $K$ et $P$. On vérifie que $f$ est régulière (si $\dim(S)=\ind(P)$). La commande "List(Basis(S))" donne une base du stabilisateur $S$. Ensuite, on calcule l'intersection de l'algèbre de Lie $S$ et son orthogonal pour la forme de Killing dans l'algèbre de Lie  $\textbf{L}$ avec la commande "Intersection(KappaPerp(L,S),S)". Remarquons que  la codimension de cet espace dans $S$ est la dimension d'un facteur réductif de $S$.\\

Si cette codimension est non nulle, l'algèbre de Lie $P$ est de rang non nul et on peut l'éliminer de la liste des sous-algèbres à traiter.\\

Dans le cas contraire , il s'avère que le choix des valeurs $a_{K}$ et $b_{L}$ que nous avons fait est suffisamment générique pour que la forme linéaire soit fortement régulière: en effet, nous sommes capables de trouver un élément $h \in \mathfrak{h}$ vérifiant l'assertion $(ii)$ du théorème \ref{théo pr} pour $f$
dans un ouvert de Zariski $A^{\prime}$ de $A$. Ceci montre que les sous-algèbres biparaboliques de rang nul et d'indice non nul ne sont pas stables et donc
que la conjecture \ref{Conj} est vraie.\\
\begin{rema}
(i) Les calculs effectués avec GAP4 ne nous permettent pas en général de déterminer le rang des sous-algèbres de Lie étudiées. Il permettent simplement de montrer que certaines d'entre-elles sont de rang non nul.\\
(ii) Néanmoins, dans la plupart des cas ces calculs montrent que, pour la forme linéaire $f$ choisie, $\mathfrak{q}(f)$ est réductif de sorte que l'indice et le rang de $\mathfrak{q}$ sont égaux et que $\mathfrak{q}$ est une algèbre de Lie quasi-réductive.
\end{rema}

\section{Résultats des calculs}\label{calculs}

Dans cette partie, nous  donnons pour chaque  algèbre de Lie simple de type exceptionnel, la liste des sous-algèbres biparaboliques $\mathfrak{q}=\mathfrak{q}_{\pi^{\prime}, \pi^{\prime\prime}}$  vérifiant la condition $(\ast)$ en indiquant leur indice et leur rang, puis, pour celles de rang nul et d'indice strictement positif, nous démontrons le théorème \ref{théo pr}. Nous n'explicitons les calculs effectués en se basant sur GAP4 que pour les types  $F_{4}$, $E_{6}$ et $E_{7}$.\\
Comme les sous-algèbres biparaboliques $\mathfrak{q}_{\pi^{\prime}, \pi^{\prime\prime}}$ et $\mathfrak{q}_{\pi^{\prime\prime}, \pi^{\prime}}$ sont conjuguées, on peut supposer que $\pi^{\prime}$ et $\pi^{\prime\prime}$ sont tels que $k_{\pi^{\prime}}+k_{\pi^{\prime\prime}}> \rang(\mathfrak{g})$ et $k_{\pi^{\prime}}> \frac{\rang(\mathfrak{g})}{2}$.\\
De plus, nous distinguerons entre les deux cas complémentaires suivants:\\
$(i)$ $\pi^{\prime}$ est connexe.\\
$(ii)$ $\pi^{\prime}$ n'est pas  connexe et $\pi^{\prime\prime}$ n'est pas connexe ou $k_{\pi^{\prime\prime}}\leq \frac{\rang(\mathfrak{g})}{2}$.\\

\textbf{Calculs dans $F_{4}$}:\\
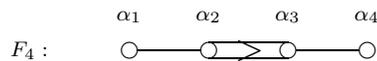
$$\begin{Dynkin}
   \Dbloc{\Dtext{l}{F_4:}}
   \Dbloc{\Dcirc \Deast \Dtext{t}{\alpha_{1}}}
   \Dbloc{\Dwest \Dcirc \Ddoubleeast  \Dtext{t}{\alpha_{2}}}
   \Drightarrow
   \Dbloc{\Ddoublewest \Dcirc \Deast \Dtext{t}{\alpha_{3}}}
   \Dbloc{\Dwest \Dcirc \Dtext{t}{\alpha_{4}}}
  \end{Dynkin} $$

Il est clair que si la sous-algèbre biparabolique $\mathfrak{q}_{\pi^{\prime}, \pi^{\prime\prime}}$ vérifie la condition $(\ast)$ alors $\pi^{\prime}$ ou $\pi^{\prime\prime}$ est de type $B_{3}$ ou $C_{3}$.\\

Supposons   $\pi^{\prime}$ de type $B_{3}$: $ \pi^{\prime}=\{\alpha_{1},\alpha_{2},\alpha_{3}\}$,  on a $\mathcal{E}_{\pi^{\prime}}=\{\alpha_{1}+2\alpha_{2}+2\alpha_{3},\alpha_{1},\alpha_{3}\}$:\\

\begin{tabular}{|l|l|c|c|c|}
\hline
&$Type$&$ \pi^{\prime\prime}$ & $\ind(\mathfrak{q}_{\pi^{\prime}, \pi^{\prime\prime}})$ & $\rang(\mathfrak{q}_{\pi^{\prime}, \pi^{\prime\prime}})$ \\
\hline
&$A_{1}^{2}$& $\{\alpha_2,\alpha_4\}$ & 1 & 0\\
\hline
&$A_{1}\times A_{2}$& $\{\alpha_1, \alpha_2,\alpha_4\}$ & 1 & 1\\
\hline
&$C_{3}$& $\{\alpha_2, \alpha_3,\alpha_4\}$ & 2 & 2\\
\hline
\end{tabular}\\
\\

Dans toute la suite et pour calculer le stabilisateur $\mathfrak{q}(f)$, on procède comme suit: on fait le calcul pour un $f\in A$ particulier, mais en fait suffisamment générique, obtenant une base de $\mathfrak{q}(f)$ dont les vecteurs s'expriment comme combinaison linéaire effective (c'est à dire dont tous les coefficients sont non nuls) de certains vecteurs radiciels, puis pour un $f$ générique quelconque on vérifie que ce résultat reste vrai.\\

D'après le tableau ci-dessus (voir les calculs ci-après), $\mathfrak{q}=\mathfrak{q}_{\{\{\alpha_{1},\alpha_{2},\alpha_{3}\}, \{\alpha_{2},\alpha_{4}\}\}}$ est la seule sous-algèbre biparabolique de rang nul et d'indice strictement positif:\\
Soit $f=K(v,.)\in A$, o\`{u} $v=a_{1}x_{-\alpha_{2}}+a_{2}x_{-\alpha_{4}}+b_{1}x_{\alpha_{1}}+b_{2}x_{\alpha_{3}}+b_{3}x_{\alpha_{1}+2\alpha_{2}+2\alpha_{3}}$ ( les $a_{i}$  et les $b_{j}$ sont tous non nuls). Dans ce cas, on a $\dim (\mathfrak{q}(f))=1$.\\
Comme $x \in \mathfrak{q}(f)$ si et seulement si $[x,v] \in \mathfrak{q}^{\perp_{K}}$, on obtient $\mathfrak{q}(f)=\mathbb{K}x$ avec
 $$x=x_{\alpha_{2}}+\lambda_{1}\,x_{-\alpha_{1}}+\lambda_{2}\,x_{-\alpha_{3}}+\lambda_{3}\, x_{-\alpha_{1}-\alpha_{2}-\alpha_{3}}+\lambda_{4}\, x_{-\alpha_{2}-2\alpha_{3}}+\lambda_{5}\, x_{-\alpha_{1}-2\alpha_{2}-2\alpha_{3}} $$
o\`{u}\\
 $$\lambda_{1}= \frac{-a_{1}}{2b_{1}}, \lambda_{2}=  \frac{-a_{1}}{2b_{2}}, \lambda_{3}= \frac{b_{2}}{b_{3}}, \lambda_{4}= \frac{b_{1}}{b_{3} } \hbox{ et } \lambda_{5}= \frac{a_{1}}{2b_{3}}.$$

Soit $h \in \mathfrak{h}$ tel que $\alpha_1(h)=\alpha_3(h)=-\alpha_2(h)=-\alpha_4(h)=1$. On a $h.f=-f$ et $[h,x]=-x$ de sorte que l'algèbre de Lie $\mathfrak{q}$ est non stable.\\

Dans cette partie, on donne les calculs effectués par GAP4  montrant que certaines  des sous-algèbres de  la liste précédente sont de rang non nul et donc sont à éliminer de la liste de celles à traiter.\\
\\
$gap>\, L:=SimpleLieAlgebra("F",4,Rationals);;R:=RootSystem(L);$\\
<root system of rank 4>\\
$gap>\, x:=PositiveRootVectors(R);;y:=NegativeRootVectors(R);;$\\
$gap>\, g:=CanonicalGenerators(R);;h:=g[3];;$\\
$gap>\,gP:=Concatenation(x[1]+x{[4]},h,y{[2..4]});;P:=Subalgebra(L,gP);;dp:=Dimension(P);$\\
15\\
$gap>\,a1:=1;;a2:=1$;;\\
$gap>\,b1:=2;;b2:=5;;b3:=1$;;\\
$gap>\, u1:=a1*y[1]+a2*y[4];;u2:=b1*x[2]+b2*x[3]+b3*x[16];;u:=u1+u2$;;\\
$gap>\,bP:=List(Basis(P));$\\
$[ v.1+v.4, v.49, v.50, v.51, v.52, v.26, v.27, v.28, v.4, v.30, v.31, v.33, v.34, v.37, v.40 ]$\\
$gap>\, bP:=List(Basis(P));;l:=[];;for \,i\, in \,[1..dp]\, do\, l[i]:=u*bP[i];od;;l;;$\\
$gap>\, V:=Subspace(L,l);;K:=KappaPerp(L,V);;S:=Intersection(K,P);;dS:=Dimension(S);$\\
1\\
$gap>\, bS:=List(Basis(S));$\\
$[ v.4+(-1/4)*v.26+(-1/10)*v.27+(5)*v.33+(2)*v.34+(1/2)*v.40 ]$\\
$gap>\, KS:=Intersection(KappaPerp(L,S),S);$\\
<vector space of dimension 1 over Rationals>\\
$gap>\,bP:=List(Basis(P));;J:=[];;for\, i \,in\, [1..dp] \,do \,J[i]:=bP[i]*bS[1];od;;J$;;\\
$gap>\,W:=Subspace(P,J);$\\
<vector space over Rationals, with 15 generators>\\
$gap>$ $A:=Intersection(S,W);$\\
<vector space of dimension 1 over Rationals>\\
\\
$gap>\,  L:=SimpleLieAlgebra("F",4,Rationals);;R:=RootSystem(L);;$\\
$gap>\, x:=PositiveRootVectors(R);;y:=NegativeRootVectors(R);;$\\
$gap>\, g:=CanonicalGenerators(R);;h:=g[3];;$\\
$gap>\, gP:=Concatenation(x{[1..2]}+x{[4]},h,y{[2..4]});;P:=Subalgebra(L,gP);;dp:=Dimension(P);$\\
17\\
$gap>\, a1:=1;;a2:=1;;$\\
$gap>\, b1:=2;;b2:=5;;b3:=1;;$\\
$gap>\, u1:=a1*y[6]+a2*y[1];;u2:=b1*x[2]+b2*x[3]+b3*x[16];;u:=u1+u2;;$\\
$gap>\, bP:=List(Basis(P));;l:=[];;for\, i\, in\, [1..dp] \,do\, l[i]:=u*bP[i];od;;l;;$\\
$gap>\, V:=Subspace(L,l);;K:=KappaPerp(L,V);;S:=Intersection(K,P);;dS:=Dimension(S);$\\
1\\
$gap>\, bS:=List(Basis(S));$\\
$[ v.2+(-1/50)*v.6+(-1/200)*v.26+(1/500)*v.27+(1/10)*v.31+(-1/100)*v.40 ]$\\
$gap>\, KS:=Intersection(KappaPerp(L,S),S);$\\
<vector space of dimension 0 over Rationals>\\
\\
$gap>\, L:=SimpleLieAlgebra("F",4,Rationals);;R:=RootSystem(L);;$\\
$gap>\, x:=PositiveRootVectors(R);;y:=NegativeRootVectors(R);;$\\
$gap>\,  g:=CanonicalGenerators(R);;h:=g[3];;$\\
$gap>\, gP:=Concatenation(x[1]+x{[3..4]},h,y{[2..4]});;P:=Subalgebra(L,gP);;dp:=Dimension(P);$\\
22\\
$\,gap> a1:=1;;a2:=1;;a3:=5;;$\\
$gap>\, b1:=2;;b2:=3;;b3:=1;;$\\
$gap>\, u1:=a1*y[4]+a2*y[10]+a3*y[15];;u2:=b1*x[2]+b2*x[3]+b3*x[16];;u:=u1+u2;;$\\
$gap> \,bP:=List(Basis(P));;l:=[];;\,for\, i\, in\, [1..dp]\, do \,l[i]:=u*bP[i];od;;l;;$\\
$gap>\, V:=Subspace(L,l);;K:=KappaPerp(L,V);;S:=Intersection(K,P);;dS:=Dimension(S);$\\
2\\
$gap>\, bS:=List(Basis(S));$\\
$[ v.3+(3/37)*v.4+(-3/37)*v.10+(-1/37)*v.27+(-6/37)*v.28+(9/37)*v.33+(6/37)*v.34, v.4+(73)*v.10+(-37/2)*v.26+(12)*v.27+(146)*v.28+(3)*v.33+(2)*v.34+(37)*v.40 ]$\\
$gap>\,  KS:=Intersection(KappaPerp(L,S),S);$\\
<vector space of dimension 0 over Rationals>\\
\\

Supposons maintenant  $\pi^{\prime}$ de type $C_{3}$: $\pi^{\prime}=\{\alpha_{2},\alpha_{3},\alpha_{4}\}$, ici on a $\mathcal{E}_{\pi^{\prime}}=\{\alpha_{2}+2\alpha_{3}+2\alpha_{4},\alpha_{2}+2\alpha_{3},\alpha_{2}\}$:\\
\\
\begin{tabular}{|l|l|l|c|c|c|}
\hline
&$Type$ & $ \pi^{\prime\prime}$ & $\ind(\mathfrak{q}_{\pi^{\prime}, \pi^{\prime\prime}})$ & $\rang(\mathfrak{q}_{\pi^{\prime}, \pi^{\prime\prime}})$ \\
\hline
&$A_{1}^{2}$& $\{\alpha_1,\alpha_3\}$ & 1 & 1\\
\hline
&$A_{1}^{2}$& $\{\alpha_1, \alpha_4\}$ & 1 & 1\\
\hline
&$A_{1}\times A_{2}$&$\{\alpha_1, \alpha_2,\alpha_4\}$ & 1 & 1\\
\hline
&$A_{1}\times A_{2}$&$\{\alpha_1, \alpha_3,\alpha_4\}$ & 1 & 1\\
\hline
&$B_{3}$&$\{\alpha_1, \alpha_2,\alpha_3\}$ & 2 & 2\\
\hline
\end{tabular}\\
\\

Dans le cas traité et en utilisant "GAP4", on vérifie que toutes  les sous-algèbres biparaboliques qui vérifient la condition $(\ast)$ sont quasi-réductives:\\
\\
$gap>\, L:=SimpleLieAlgebra("F",4,Rationals);;R:=RootSystem(L);;$\\
$gap>\, x:=PositiveRootVectors(R);;y:=NegativeRootVectors(R);;$\\
$gap>\, g:=CanonicalGenerators(R);;h:=g[3];;$\\
$gap>\, gP:=Concatenation(x{[2..3]},h,y[1]+y{[3..4]});;P:=Subalgebra(L,gP);;dp:=Dimension(P);$\\
15\\
$gap>\, a1:=1;;a2:=1;;$\\
$gap>\,  b1:=2;;b2:=5;;b3:=1;;$\\
$gap>\, u1:=a1*y[2]+a2*y[3];;u2:=b1*x[4]+b2*x[10]+b3*x[15];;u:=u1+u2;;$\\
$gap>\, bP:=List(Basis(P));;l:=[];;for\, i\, in \,[1..dp]\, do\, l[i]:=u*bP[i];od;;l;;$\\
$gap>\, V:=Subspace(L,l);;K:=KappaPerp(L,V);;S:=Intersection(K,P);;dS:=Dimension(S);$\\
1\\
$gap>\,bS:=List(Basis(S));$\\
$[ v.3+(-2/5)*v.27+(-1/2)*v.28+(1/5)*v.34 ]$\\
$gap>\,  KS:=Intersection(KappaPerp(L,S),S);$\\
<vector space of dimension 0 over Rationals>\\
\\
$gap>\, L:=SimpleLieAlgebra("F",4,Rationals);;R:=RootSystem(L);;$\\
$gap>\,  x:=PositiveRootVectors(R);;y:=NegativeRootVectors(R);;$\\
$gap>\,  g:=CanonicalGenerators(R);;h:=g[3];;$\\
$gap>\, gP:=Concatenation(x{[1..2]},h,y[1]+y{[3..4]});;P:=Subalgebra(L,gP);;dp:=Dimension(P);$\\
15\\
$gap>\, a1:=1;;a2:=1;;$\\
$gap>\, b1:=2;;b2:=5;;b3:=1;;$\\
$gap>\,  u1:=a1*y[2]+a2*y[1];;u2:=b1*x[4]+b2*x[10]+b3*x[15];;u:=u1+u2;;$\\
$gap>\,  bP:=List(Basis(P));;l:=[];;for\, i \,in\, [1..dp]\, do\, l[i]:=u*bP[i];od;;l;;$\\
$gap>\, V:=Subspace(L,l);;K:=KappaPerp(L,V);;S:=Intersection(K,P);;dS:=Dimension(S);$\\
1\\
$gap>\, bS:=List(Basis(S));$\\
$[ v.1+(-5)*v.25+(-1/5)*v.34+v.39 ]$\\
$gap>\,  KS:=Intersection(KappaPerp(L,S),S);$\\
<vector space of dimension 0 over Rationals>\\
\\
$gap>\, L:=SimpleLieAlgebra("F",4,Rationals);;R:=RootSystem(L);;$\\
$gap>\, x:=PositiveRootVectors(R);;y:=NegativeRootVectors(R);;$\\
$gap>\,  g:=CanonicalGenerators(R);;h:=g[3];;$\\
$gap>\,  gP:=Concatenation(x{[1..2]}+x{[4]},h,y[1]+y{[3..4]});;P:=Subalgebra(L,gP);;dp:=Dimension(P);$\\
17\\
$gap>\, a1:=1;;a2:=1;;$\\
$gap>\, b1:=2;;b2:=5;;b3:=1;;$\\
$gap>\,  u1:=a1*y[6]+a2*y[1];;u2:=b1*x[4]+b2*x[10]+b3*x[15];;u:=u1+u2;$\\
$gap>\, bP:=List(Basis(P));;l:=[];;for\, i\, in \,[1..dp] \,do\, l[i]:=u*bP[i];od;;l;;$\\
$gap>\,  V:=Subspace(L,l);;K:=KappaPerp(L,V);;S:=Intersection(K,P);;dS:=Dimension(S);$\\
1\\
$gap>\, bS:=List(Basis(S));$\\
$[ v.1+(-5)*v.25+(-1/5)*v.34+v.39 ]$\\
$gap>\, KS:=Intersection(KappaPerp(L,S),S);$\\
<vector space of dimension 0 over Rationals>\\
\\
$gap>\, L:=SimpleLieAlgebra("F",4,Rationals);;R:=RootSystem(L);;$\\
$gap>\,  x:=PositiveRootVectors(R);;y:=NegativeRootVectors(R);;$\\
$gap> \, g:=CanonicalGenerators(R);;h:=g[3];;$\\
$gap>\, gP:=Concatenation(x{[1..3]},h,y[1]+y{[3..4]});;P:=Subalgebra(L,gP);;dp:=Dimension(P);$\\
17\\
$gap>\, a1:=1;;a2:=1;;$\\
$gap>\,  b1:=2;;b2:=5;;b3:=1;;$\\
$gap> \,u1:=a1*y[5]+a2*y[2];;u2:=b1*x[4]+b2*x[10]+b3*x[15];;u:=u1+u2;;$\\
$gap>\, bP:=List(Basis(P));;l:=[];;for\, i\, in\, [1..dp] \,do\, l[i]:=u*bP[i];od;;l;;$\\
$gap>\, V:=Subspace(L,l);;K:=KappaPerp(L,V);;S:=Intersection(K,P);;dS:=Dimension(S);$\\
1\\
$gap>\,bS:=List(Basis(S));$\\
$[ v.5+(-1/2)*v.28+(-2)*v.29+v.39 ]$
$gap>\, KS:=Intersection(KappaPerp(L,S),S);$\\
<vector space of dimension 0 over Rationals>\\

\textbf{Calculs dans $E_{6}$}:\\
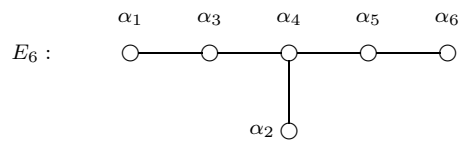
$$\begin{Dynkin}
\Dbloc{\Dtext{l}{E_6:}}
\Dbloc{\Dcirc \Deast\Dtext{t}{\alpha_{1}} }
 \Dbloc{\Dwest\Dcirc \Deast\Dtext{t}{\alpha_{3}} }
 \Dbloc{\Dwest \Dcirc \Dsouth \Deast\Dtext{t}{\alpha_{4}} }
 \Dbloc{\Dwest \Dcirc \Deast\Dtext{t}{\alpha_{5}}}
\Dbloc{ \Dwest \Dcirc\Dtext{t}{\alpha_{6}} }
\Dskip
\Dspace \Dspace \Dspace \Dbloc{\Dnorth \Dcirc\Dtext{l}{\alpha_{2}}}
\end{Dynkin}$$

Il est clair que si la sous-algèbre biparabolique $\mathfrak{q}_{\pi^{\prime}, \pi^{\prime\prime}}$ vérifie la condition $(\ast)$ alors $\pi^{\prime}$ ou $\pi^{\prime\prime}$ est de type $D_{4}$ ou $D_{5}$.\\
\\

Supposons  $\pi^{\prime}$ de type $D_{4}$, on a $\pi^{\prime}=\{\alpha_{2},\alpha_{3},\alpha_{4},\alpha_{5}\}$. Ici, on a $\mathcal{E}_{\pi^{\prime}}=\{\alpha_{2}+\alpha_{3}+2\alpha_{4}+\alpha_{5},\alpha_{2},\alpha_{3},\alpha_{5}\}$:\\
\\
\begin{tabular}{|l|l|l|c|c|c|}
\hline
&$Type$ & $ \pi^{\prime\prime}$ & $\ind(\mathfrak{q}_{\pi^{\prime}, \pi^{\prime\prime}})$ & $\rang(\mathfrak{q}_{\pi^{\prime}, \pi^{\prime\prime}})$ \\
\hline
&$A_{1}^{3}$& $\{\alpha_1,\alpha_4,\alpha_6\}$ & 1 & 0\\
\hline
&$A_{1}^{2}\times A_{2}$& $\{\alpha_1,\alpha_{2}, \alpha_4,\alpha_6\}$ & 1 & 1\\
\hline
&$A_{1}\times A_{4}$&$\{\alpha_1, \alpha_2,\alpha_3,\alpha_4,\alpha_6\}$ & 1 & 1\\
\hline
&$A_{1}\times A_{4}$&$\{\alpha_1, \alpha_2,\alpha_4,\alpha_5,\alpha_6\}$ & 1 & 1\\
\hline
\end{tabular}\\
\\

Dans ce cas et d'après le tableau ci-dessus (voir les calculs ci-après), $$\mathfrak{q}=\mathfrak{q}_{\{\{\alpha_{2},\alpha_{3},\alpha_{5},\alpha_{4}\},\{\alpha_{1},\alpha_{4},\alpha_{6}\}\}}$$ est la seule sous-algèbre biparabolique de rang nul et d'indice strictement positif:\\
Soit $f=K(v,.)\in A$, o\`{u} $v=a_{1}\,x_{-\alpha_{1}}+a_{2}\,x_{-\alpha_{4}}+a_{3}\,x_{-\alpha_{6}}+b_{1}\,x_{\alpha_{2}}+b_{2}\,
x_{\alpha_{3}}+b_{3}\,x_{\alpha_{5}}+b_{4}\,x_{\alpha_{2}+\alpha_{3}+2\alpha_{4}+\alpha_{5}}$ (les $a_{i}$  et les $b_{j}$ sont tous non nuls). Dans ce cas, on a $\dim (\mathfrak{q}(f))=1$.\\
Comme $x \in \mathfrak{q}(f)$ si et seulement si $[x,v] \in \mathfrak{q}^{\perp_{K}}$, on obtient $\mathfrak{q}(f)=\mathbb{K}x$ avec\\
 $x=x_{\alpha_{4}}+\lambda_{1}\,x_{-\alpha_{2}}+\lambda_{2}\,x_{-\alpha_{3}}+\lambda_{3}\, x_{-\alpha_{5}}+\lambda_{4}\, x_{-\alpha_{2}-\alpha_{3}-\alpha_{4}}+\lambda_{5}\,x_{-\alpha_{2}-\alpha_{4}-\alpha_{5}} +\lambda_{6}\,x_{-\alpha_{3}-\alpha_{4}-\alpha_{5}}+\lambda_{7}\,x_{-\alpha_{2}-\alpha_{3}-2\alpha_{4}-\alpha_{5}}.$\\
o\`{u}\\
$ \lambda_{1}= \frac{-a_{2}}{2b_{1}}$, $\lambda_{2}= \frac{-a_{2}}{2b_{2}}$, $\lambda_{3}= \frac{a_{2}}{2b_{3}}$, $\lambda_{4}= \frac{b_{3}}{b_{4}}$, $\lambda_{5}= \frac{-b_{2}}{b_{4}}$, $\lambda_{6}= \frac{-b_{1}}{b_{4}}$ et  $\lambda_{7}= \frac{a_{2}}{2b_{4}}$.\\

Soit $h \in \mathfrak{h}$ tel que  $\alpha_2(h)=\alpha_3(h)=\alpha_5(h)=-\alpha_1(h)=-\alpha_4(h)=-\alpha_6(h)=1$. On a $h.f=-f$ et  $[h,x]=-x$ de sorte que l'algèbre de Lie $\mathfrak{q}$ est non stable.\\

Dans cette partie, on donne les calculs effectués par GAP4  montrant que certaines  des sous-algèbres de  la liste sont de rang non nul et donc sont à éliminer de la liste de celles à traiter.\\
\\
$gap> L:=SimpleLieAlgebra("E",6,Rationals);;R:=RootSystem(L);;$\\
$gap> x:=PositiveRootVectors(R);;y:=NegativeRootVectors(R);;$\\
$gap> g:=CanonicalGenerators(R);;h:=g[3];;$\\
$gap> gP:=Concatenation(x{[1]}+x{[4]}+x{[6]},h,y{[2..5]});;P:=Subalgebra(L,gP);;dp:=Dimension(P);$\\
$21$\\
$gap> a1:=1;;a2:=1;;a3:=3;;a4:=2;;$\\
$gap> b1:=3;;b2:=13;;b3:=1;;$\\
$gap> u1:=b1*y[1]+b2*y[4]+b3*y[6];;u2:=a1*x[2]+a2*x[3]+a3*x[5]+a4*x[24];;u:=u1+u2;;$\\
$gap> bP:=List(Basis(P));;l:=[];;for i in [1..dp] do l[i]:=u*bP[i];od;;l;;$\\
$gap> V:=Subspace(L,l);;K:=KappaPerp(L,V);;S:=Intersection(K,P);;dS:=Dimension(S);$\\
$1$\\
$gap> bS:=List(Basis(S));$\\
$[ v.4+(-13/2)*v.38+(-13/2)*v.39+(-13/6)*v.41+(3/2)*v.49+(-1/2)*v.50+(-1/2)*v.51+(13/4)*v.60 ]$\\
$gap> KS:=Intersection(KappaPerp(P,S),S);$\\
$<vector space of dimension 1 over Rationals>$\\

$gap> L:=SimpleLieAlgebra("E",6,Rationals);;R:=RootSystem(L);;$\\
$gap> x:=PositiveRootVectors(R);;y:=NegativeRootVectors(R);;$\\
$gap> g:=CanonicalGenerators(R);;h:=g[3];;$\\
$gap> gP:=Concatenation(x{[1..2]}+x{[4]}+x{[6]},h,y{[2..5]});;P:=Subalgebra(L,gP);;dp:=Dimension(P);$\\
$23$\\
$gap> a1:=1;;a2:=1;;a3:=3;;a4:=2;;$\\
$gap> b1:=3;;b2:=13;;b3:=1;;$\\
$gap> u1:=b1*y[1]+b2*y[8]+b3*y[6];;u2:=a1*x[2]+a2*x[3]+a3*x[5]+a4*x[24];;u:=u1+u2;;$\\
$gap>  bP:=List(Basis(P));;l:=[];;for i in [1..dp] do l[i]:=u*bP[i];od;;l;;$\\
$gap> V:=Subspace(L,l);;K:=KappaPerp(L,V);;S:=Intersection(K,P);;dS:=Dimension(S);$\\
$1$\\
$gap> bS:=List(Basis(S));$\\
$[ v.2+(-13/3)*v.8+(-169/6)*v.38+(169/6)*v.39+(169/18)*v.41+(13/2)*v.45+(-13/6)*v.46+(-169/12)*v.60 ]$\\
$gap> KS:=Intersection(KappaPerp(P,S),S);$\\
$<vector space of dimension 0 over Rationals>$\\
\\
$gap> L:=SimpleLieAlgebra("E",6,Rationals);;R:=RootSystem(L);;$\\
$gap> x:=PositiveRootVectors(R);;y:=NegativeRootVectors(R);;$\\
$gap> g:=CanonicalGenerators(R);;h:=g[3];;$\\
$gap> gP:=Concatenation(x{[1..4]}+x{[6]},h,y{[2..5]});;P:=Subalgebra(L,gP);;dp:=Dimension(P);$\\
$29$\\
$gap> a1:=1;;a2:=1;;a3:=3;;a4:=2;;$\\
$gap> b1:=3;;b2:=13;;b3:=1;;$\\
$gap> u1:=b1*y[9]+b2*y[17]+b3*y[6];;u2:=a1*x[2]+a2*x[3]+a3*x[5]+a4*x[24];;u:=u1+u2;;$\\
$gap>  bP:=List(Basis(P));;l:=[];;for i in [1..dp] do l[i]:=u*bP[i];od;;l;;$\\
$gap> V:=Subspace(L,l);;K:=KappaPerp(L,V);;S:=Intersection(K,P);;dS:=Dimension(S);$\\
$1$\\
$gap> bS:=List(Basis(S));$\\
$[ v.1+(-26/9)*v.3+(26/9)*v.9+(-13/3)*v.38+(13/3)*v.39+(-13/9)*v.41+(-13/3)*v.44+(13/9)*v.46+(13/6)*v.60 ]$\\
$gap> KS:=Intersection(KappaPerp(P,S),S);$\\
$<vector space of dimension 0 over Rationals>$\\

$gap> L:=SimpleLieAlgebra("E",6,Rationals);;R:=RootSystem(L);;$\\
$gap> x:=PositiveRootVectors(R);;y:=NegativeRootVectors(R);;$\\
$gap> g:=CanonicalGenerators(R);;h:=g[3];;$\\
$gap> gP:=Concatenation(x{[1..2]}+x{[4..6]},h,y{[2..5]});;P:=Subalgebra(L,gP);;dp:=Dimension(P);$\\
$29$\\
$gap> a1:=1;;a2:=1;;a3:=3;;a4:=2;;$\\
$gap> b1:=3;;b2:=13;;b3:=1;;$\\
$gap> u1:=b1*y[1]+b2*y[10]+b3*y[20];;u2:=a1*x[2]+a2*x[3]+a3*x[5]+a4*x[24];;u:=u1+u2;;$\\
$gap>  bP:=List(Basis(P));;l:=[];;for i in [1..dp] do l[i]:=u*bP[i];od;;l;;$\\
$gap>  V:=Subspace(L,l);;K:=KappaPerp(L,V);;S:=Intersection(K,P);;dS:=Dimension(S);$\\
$1$\\
$gap> bS:=List(Basis(S));$\\
$[ v.5+(2197/2)*v.6+(-13)*v.10+(169/2)*v.38+(169/2)*v.39+(-169/6)*v.41+(-13/2)*v.44+(-13/2)*v.45+(-169/4)*v.60 ]$\\
$gap> KS:=Intersection(KappaPerp(P,S),S);$\\
$<vector space of dimension 0 over Rationals>$\\

Supposons maintenant   $\pi^{\prime}$ de type $D_{5}$, par symétrie on prend  $\pi^{\prime}=\{\alpha_{1},\alpha_{2},\alpha_{3},\alpha_{4},\alpha_{5}\}$. Ici, on a $\mathcal{E}_{\pi^{\prime}}=\{\alpha_{1}+\alpha_{2}+2\alpha_{3}+2\alpha_{4}+\alpha_{5},\alpha_{2}+\alpha_{4}+\alpha_{5},\alpha_{1},\alpha_{4}\}$:\\
\\
\begin{tabular}{|l|l|l|c|c|c|}
\hline
&$Type$ & $ \pi^{\prime\prime}$ & $\ind(\mathfrak{q}_{\pi^{\prime}, \pi^{\prime\prime}})$ & $\rang(\mathfrak{q}_{\pi^{\prime}, \pi^{\prime\prime}})$ \\
\hline
&$A_{1}^{3}$& $\{\alpha_2,\alpha_3,\alpha_6\}$ & 1 & 0\\
\hline
&$A_{1}^{2}\times A_{2}$& $\{\alpha_1,\alpha_{2}, \alpha_3,\alpha_6\}$ & 1 & 1\\
\hline
&$A_{1}^{2}\times A_{2}$&$\{\alpha_2, \alpha_3,\alpha_5,\alpha_6\}$ & 1 & 0\\
\hline
&$A_{1}\times A_{2}^{2}$&$\{\alpha_1, \alpha_2,\alpha_3,\alpha_5,\alpha_6\}$ & 1 & 1\\
\hline
&$A_{1}\times A_{4}$&$\{\alpha_1, \alpha_2,\alpha_3,\alpha_4,\alpha_6\}$ & 1 & 1\\
\hline
\end{tabular}\\
\\

Dans ce cas et d'après le tableau ci-dessus (voir les calculs ci-après), les sous-algèbres biparaboliques $\mathfrak{q}_{\{\{\alpha_{1},\alpha_{2},\alpha_{3},\alpha_{4},\alpha_{5}\},\{\alpha_{2},\alpha_{3},\alpha_{6}\}\}}$ et $\mathfrak{q}_{\{\{\alpha_{1},\alpha_{2},\alpha_{3},\alpha_{4},\alpha_{5}\},\{\alpha_{2},\alpha_{3},\alpha_{5},\alpha_{6}\}\}}$ sont de rang nul et d'indice strictement positif.\\

Soient $\mathfrak{q}=\mathfrak{q}_{\{\{\alpha_{1},\alpha_{2},\alpha_{3},\alpha_{4},\alpha_{5}\},\{\alpha_{2},\alpha_{3},\alpha_{6}\}\}}$ et $f=K(v,.)\in A$, o\`{u} $v=a_{1}\,x_{-\alpha_{2}}+a_{2}\,x_{-\alpha_{3}}+a_{3}\,x_{-\alpha_{6}}+b_{1}\,x_{\alpha_{1}}+b_{2}\,
x_{\alpha_{4}}+b_{3}\,x_{\alpha_{2}+\alpha_{4}+\alpha_{5}}+b_{4}\,x_{\alpha_{1}+\alpha_{2}+2\alpha_{3}+2\alpha_{4}+\alpha_{5}}$ (les $a_{i}$  et les $b_{j}$ sont tous non nuls). Dans ce cas, on a $\dim (\mathfrak{q}(f))=1$.\\

Comme $x \in \mathfrak{q}(f)$ si et seulement si $[x,v] \in \mathfrak{q}^{\perp_{K}}$, on obtient $\mathfrak{q}(f)=\mathbb{K}x$ avec\\
 $x=x_{\alpha_{3}}+\lambda_{1}\,x_{-\alpha_{1}}+\lambda_{2}\,x_{-\alpha_{4}}+\lambda_{3}\, x_{-\alpha_{1}-\lambda_{3}-\alpha_{4}}\,+ \lambda_{4}\, x_{-\alpha_{2}-\alpha_{4}-\alpha_{5}}+\lambda_{5}\,x_{-\alpha_{1}-\alpha_{2}-\alpha_{3}-\alpha_{4}-\alpha_{5}} +\lambda_{6}\,x_{-\alpha_{2}-\alpha_{3}-2\alpha_{4}-\alpha_{5}}+\lambda_{7}\,x_{-\alpha_{1}-\alpha_{2}-2\alpha_{3}-2\alpha_{4}-\alpha_{5}}.$\\
o\`{u}\\
$\lambda_{1}=\frac{-a_{2}}{2b_{1}}$, $\lambda_{2}= \frac{-a_{2}}{2b_{2}}$, $\lambda_{3}= \frac{-b_{3}}{b_{4}}$, $\lambda_{4}= \frac{-a_{2}}{2b_{3}}$, $\lambda_{5}= \frac{-b_{2}}{b_{4}}$, $\lambda_{6}= \frac{b_{1}}{b_{4}}$ et $\lambda_{7}= \frac{a_{2}}{2b_{4}}$.\\

Soit $h \in \mathfrak{h}$ tel que  $\alpha_1(h)=\alpha_4(h)=\alpha_5(h)=1=-\alpha_2(h)=-\alpha_3(h)=-\alpha_6(h)=1$. On a $h.f=-f$ et $[h,x]=-x$ de sorte que l'algèbre de Lie $\mathfrak{q}$ est non stable.\\

Soient maintenant $\mathfrak{q}=\mathfrak{q}_{\{\{\alpha_{1},\alpha_{2},\alpha_{3},\alpha_{4},\alpha_{5}\},\{\alpha_{2},\alpha_{3},\alpha_{5},\alpha_{6}\}\}}$
et $f=K(v,.)\in A$, o\`{u} $v=a_{1}\,x_{-\alpha_{2}}+a_{2}\,x_{-\alpha_{3}}+a_{3}\,x_{-\alpha_{5}-\alpha_{6}}+b_{1}\,x_{\alpha_{1}}+b_{2}\,
x_{\alpha_{4}}+b_{3}\,x_{\alpha_{2}+\alpha_{4}+\alpha_{5}}+b_{4}\,x_{\alpha_{1}+\alpha_{2}+2\alpha_{3}+2\alpha_{4}+\alpha_{5}}$ (les $a_{i}$  et les $b_{j}$ sont tous non nuls). Dans ce cas, on a $\dim (\mathfrak{q}(f))=1$.\\

Comme $x \in \mathfrak{q}(f)$ si et seulement si $[x,v] \in \mathfrak{q}^{\perp_{K}}$, on obtient $\mathfrak{q}(f)=\mathbb{K}x$ avec\\
 $x=x_{\alpha_{3}}+\lambda_{1}\,x_{-\alpha_{1}}+\lambda_{2}\,x_{-\alpha_{4}}+\lambda_{3}\, x_{-\alpha_{1}-\lambda_{3}-\alpha_{4}}\,+ \lambda_{4}\, x_{-\alpha_{2}-\alpha_{4}-\alpha_{5}}+\lambda_{5}\,x_{-\alpha_{1}-\alpha_{2}-\alpha_{3}-\alpha_{4}-\alpha_{5}} +\lambda_{6}\,x_{-\alpha_{2}-\alpha_{3}-2\alpha_{4}-\alpha_{5}}+\lambda_{7}\,x_{-\alpha_{1}-\alpha_{2}-2\alpha_{3}-2\alpha_{4}-\alpha_{5}}.$\\
o\`{u}\\
     $\lambda_{1}= \frac{-a_{2}}{2b_{1}}$, $\lambda_{2}= \frac{-a_{2}}{2b_{2}}$, $\lambda_{3}= \frac{-b_{3}}{b_{4}}$, $\lambda_{4}=\frac{-a_{2}}{2b_{3}}$, $\lambda_{5}= \frac{-b_{2}}{b_{4}}$, $\lambda_{6}= \frac{b_{1}}{b_{4}}$ et $\lambda_{7}=\frac{a_{2}}{2b_{4}}$.\\

Soit $h \in \mathfrak{h}$ tel que  $\alpha_1(h)=\alpha_4(h)=\alpha_5(h)=-\alpha_2(h)=-\alpha_3(h)=-1$ et $\alpha_6(h)=-2$. On a $h.f=-f$ et $[h,x]=-x$ de sorte que l'algèbre de Lie $\mathfrak{q}$ est non stable.\\

Dans cette partie, on donne les calculs effectués par GAP4  montrant que certaines  des sous-algèbres de  la liste sont de rang non nul et donc sont à éliminer de la liste de celles à traiter.\\
\\
$gap>  L:=SimpleLieAlgebra("E",6,Rationals);;R:=RootSystem(L);;$\\
$gap>  x:=PositiveRootVectors(R);;y:=NegativeRootVectors(R);;$\\
$gap> g:=CanonicalGenerators(R);;h:=g[3];;$\\
$gap>  gP:=Concatenation(x{[2..3]}+[x[6]],h,y{[1..5]});;P:=Subalgebra(L,gP);;dp:=Dimension(P);$\\
$29$\\
$gap> a1:=1;;a2:=1;;a3:=3;;a4:=2;;$\\
$gap> b1:=3;;b2:=13;;b3:=1;;$\\
$gap> u1:=b1*y[2]+b2*y[3]+b3*y[6];;u2:=a1*x[1]+a2*x[4]+a3*x[14]+a4*x[29];;u:=u1+u2;;$\\
$gap> bP:=List(Basis(P));;l:=[];;for i in [1..dp] do l[i]:=u*bP[i];od;;l:;$\\
$gap>  V:=Subspace(L,l);;K:=KappaPerp(L,V);;S:=Intersection(K,P);;dS:=Dimension(S);$\\
$1$\\
$gap> KS:=Intersection(KappaPerp(L,S),S);$\\
$<vector space of dimension 1 over Rationals>$\\
$gap> bS:=List(Basis(S));$\\
$[ v.3+(-13/2)*v.37+(-13/2)*v.40+(-3/2)*v.48+(-13/6)*v.50+(-1/2)*v.58+(1/2)*v.60+(13/4)*v.65 ]$\\
\\
$gap> L:=SimpleLieAlgebra("E",6,Rationals);;R:=RootSystem(L);;$\\
$gap> x:=PositiveRootVectors(R);;y:=NegativeRootVectors(R);;$\\
$gap> g:=CanonicalGenerators(R);;h:=g[3];;$\\
$gap> gP:=Concatenation(x{[2..3]}+x{[5..6]},h,y{[1..5]});;P:=Subalgebra(L,gP);;dp:=Dimension(P);$\\
$31$\\
$gap> a1:=1;;a2:=1;;a3:=3;;a4:=2;;$\\
$gap> b1:=3;;b2:=13;;b3:=1;;$\\
$gap> u1:=b1*y[2]+b2*y[3]+b3*y[11];;u2:=a1*x[1]+a2*x[4]+a3*x[14]+a4*x[29];;u:=u1+u2;;$\\
$gap> bP:=List(Basis(P));;l:=[];;for i in [1..dp] do l[i]:=u*bP[i];od;;l;;$\\
$gap>  V:=Subspace(L,l);;K:=KappaPerp(L,V);;S:=Intersection(K,P);;dS:=Dimension(S);$\\
$1$\\
$gap> KS:=Intersection(KappaPerp(L,S),S);$\\
$<vector space of dimension 1 over Rationals>$\\
$gap>  bS:=List(Basis(S));$\\
$[ v.3+(-13/2)*v.37+(-13/2)*v.40+(-3/2)*v.48+(-13/6)*v.50+(-1/2)*v.58+(1/2)*v.60+(13/4)*v.65 ]$\\*
\\
$gap> L:=SimpleLieAlgebra("E",6,Rationals);;R:=RootSystem(L);;$\\
$gap> x:=PositiveRootVectors(R);;y:=NegativeRootVectors(R);;$\\
$gap> g:=CanonicalGenerators(R);;h:=g[3];;$\\
$gap>  gP:=Concatenation(x{[1..3]}+x{[5..6]},h,y{[1..5]});;P:=Subalgebra(L,gP);;dp:=Dimension(P);$\\
$33$\\
$gap> a1:=1;;a2:=1;;a3:=3;;a4:=2;;$\\
$gap>  b1:=3;;b2:=13;;b3:=1;;$\\
$gap>  u1:=b1*y[2]+b2*y[7]+b3*y[11];;u2:=a1*x[1]+a2*x[4]+a3*x[14]+a4*x[29];;u:=u1+u2;;$\\
$gap>  bP:=List(Basis(P));;l:=[];;for i in [1..dp] do l[i]:=u*bP[i];od;;l;;$\\
$gap> V:=Subspace(L,l);;K:=KappaPerp(L,V);;S:=Intersection(K,P);;dS:=Dimension(S);$\\
$1$\\
$gap> KS:=Intersection(KappaPerp(L,S),S);$\\
$<vector space of dimension 0 over Rationals>$\\
\\
$gap> L:=SimpleLieAlgebra("E",6,Rationals);;R:=RootSystem(L);;$\\
$gap> x:=PositiveRootVectors(R);;y:=NegativeRootVectors(R);;$\\
$gap> g:=CanonicalGenerators(R);;h:=g[3];;$\\
$gap> gP:=Concatenation(x{[1..4]}+x{[6]},h,y{[1..5]});;P:=Subalgebra(L,gP);;dp:=Dimension(P);$\\
$37$\\
$gap> a1:=1;;a2:=-1;;a3:=1;;a4:=-1;;$\\
$gap>  b1:=3;;b2:=-2;;b3:=1;;$\\
$gap> u1:=b1*y[9]+b2*y[17]+b3*y[6];;u2:=a1*x[1]+a2*x[4]+a3*x[14]+a4*x[29];;u:=u1+u2;;$\\
$gap> bP:=List(Basis(P));;l:=[];;for i in [1..dp] do l[i]:=u*bP[i];od;;l;;$\\
$gap> V:=Subspace(L,l);;K:=KappaPerp(L,V);;S:=Intersection(K,P);;dS:=Dimension(S);$\\
$1$
$gap> KS:=Intersection(KappaPerp(L,S),S);$\\
$<vector space of dimension 0 over Rationals>$\\
\\
$gap> L:=SimpleLieAlgebra("E",6,Rationals);;R:=RootSystem(L);;$\\
$gap> x:=PositiveRootVectors(R);;y:=NegativeRootVectors(R);;$\\
$gap> g:=CanonicalGenerators(R);;h:=g[3];;$\\
$gap> gP:=Concatenation(x{[1..3]}+x{[6]},h,y{[1..5]});;P:=Subalgebra(L,gP);;dp:=Dimension(P);$\\
$31$\\
$gap> a1:=1;;a2:=-1;;a3:=1;;a4:=-1;;$\\
$gap>  b1:=3;;b2:=-2;;b3:=1;;$\\
$gap> u1:=b1*y[7]+b2*y[2]+b3*y[6];;u2:=a1*x[1]+a2*x[4]+a3*x[14]+a4*x[29];;u:=u1+u2;;$\\
$gap> bP:=List(Basis(P));;l:=[];;for i in [1..dp] do l[i]:=u*bP[i];od;;l;;$\\
\\
$gap>  V:=Subspace(L,l);;K:=KappaPerp(L,V);;S:=Intersection(K,P);;dS:=Dimension(S);$\\
$1$\\
$gap> KS:=Intersection(KappaPerp(L,S),S);$\\
$<vector space of dimension 0 over Rationals>$\\
\\
\\

\textbf{Calculs dans $E_{7}$ }:\\

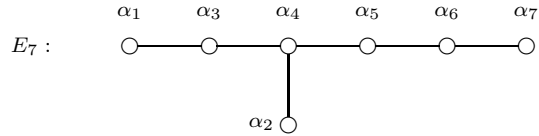
$$\begin{Dynkin}
\Dbloc{\Dtext{l}{E_7:}}
\Dbloc{\Dcirc \Deast\Dtext{t}{\alpha_{1}} }
 \Dbloc{\Dwest\Dcirc \Deast\Dtext{t}{\alpha_{3}} }
 \Dbloc{\Dwest \Dcirc \Dsouth \Deast\Dtext{t}{\alpha_{4}} }
 \Dbloc{\Dwest \Dcirc \Deast\Dtext{t}{\alpha_{5}}}
\Dbloc{ \Dwest  \Deast \Dcirc\Dtext{t}{\alpha_{6}} }
\Dbloc{ \Dwest \Dcirc\Dtext{t}{\alpha_{7}} }
\Dskip
\Dspace \Dspace \Dspace \Dbloc{\Dnorth \Dcirc\Dtext{l}{\alpha_{2}}}
\end{Dynkin}$$
\\

Remarquons tout d'abord que si  la sous-algèbre biparabolique $\mathfrak{q}_{\pi^{\prime}, \pi^{\prime\prime}}$ vérifie la condition $(\ast)$ alors $\pi^{\prime}$ est tel que $\mathcal{K}(\pi^{\prime})$ est de cardinal supérieur ou égal à 4. \\
\\
\textbf{Cas o\`{u} $\pi^{\prime}$  est connexe: }\\

Alors $\pi^{\prime}$ peut être de type $D_{4}$, $D_{5}$, $D_{6}$ ou $E_{6}$. On vérifie que si $\pi^{\prime}$ est de type $D_{4}$ la condition $(\ast)$ ne peut être satisfaite.\\

Supposons    $\pi^{\prime}$ de type $D_{5}$,   $\pi^{\prime}=\{\alpha_{1},\alpha_{2},\alpha_{3},\alpha_{4},\alpha_{5}\}$. Ici, on a $\mathcal{E}_{\pi^{\prime}}=\{\alpha_{1}+\alpha_{2}+2\alpha_{3}+2\alpha_{4}+\alpha_{5},\alpha_{2}+\alpha_{4}+\alpha_{5},\alpha_{1},\alpha_{4}\}$:\\
\\
\begin{tabular}{|l|l|c|c|c|}
\hline
&$Type$ & $ \pi^{\prime\prime}$ & $\ind(\mathfrak{q}_{\pi^{\prime}, \pi^{\prime\prime}})$ & $\rang(\mathfrak{q}_{\pi^{\prime}, \pi^{\prime\prime}})$ \\
\hline
&$A_{1}^{2}\times A_{3}$& $\{\alpha_2,\alpha_3,\alpha_{5},\alpha_6,\alpha_{7}\}$ & 1 & 0\\
\hline
&$A_{1}\times A_{2}\times A_{3}$& $\{\alpha_1,\alpha_{2}, \alpha_3,\alpha_{5},\alpha_6,\alpha_{7}\}$ & 1 & 1\\
\hline
&$D_{6}$&$\{\alpha_2, \alpha_3,\alpha_{4},\alpha_5,\alpha_6,\alpha_{7}\}$ & 3 & 3\\
\hline
\end{tabular}\\
\\

Dans ce cas et d'après le tableau ci-dessus, $\mathfrak{q}=\mathfrak{q}_{\{\{\alpha_1,\alpha_{2},\alpha_{3},\alpha_{4},\alpha_{5}\},\{\alpha_{2},\alpha_{3},\alpha_5,\alpha_{6},\alpha_7\}\}}$ est la seule sous-algèbre biparabolique de rang nul et d'indice strictement positif:\\
Soit $f=K(v,.)\in A$, o\`{u} $v=a_{1}\,x_{-\alpha_{2}}+a_{2}\,x_{-\alpha_{3}}+a_{3}\,x_{-\alpha_5-\alpha_{6}-\alpha_7}+a_{4}\,x_{-\alpha_{6}}+b_{1}\,x_{\alpha_{1}}+b_{2}\,
x_{\alpha_{4}}+b_{3}\,x_{\alpha_2+\alpha_4+\alpha_{5}}+b_{4}\,x_{\alpha_1+\alpha_{2}+2\alpha_{3}+2\alpha_{4}+\alpha_{5}}$ (les $a_{i}$  et les $b_{j}$ sont tous non nuls). Dans ce cas, on a $\mathfrak{q}(f)=\mathbb{K}x$ avec\\
 $x=x_{\alpha_{3}}+\lambda_{1}\,x_{-\alpha_{1}}+\lambda_{2}\,x_{-\alpha_{4}}+\lambda_{3}\, x_{-\alpha_1-\alpha_3-\alpha_{4}}+\lambda_{4}\, x_{-\alpha_{2}-\alpha_{4}-\alpha_{5}}+\lambda_{5}\,x_{-\alpha_1-\alpha_{2}-\alpha_3-\alpha_{4}-\alpha_{5}} +\lambda_{6}\,x_{-\alpha_{2}-\alpha_3-2\alpha_{4}-\alpha_{5}}+\lambda_{7}\,x_{-\alpha_1-\alpha_{2}-2\alpha_{3}-2\alpha_{4}-\alpha_{5}},$\\
\\
o\`{u} $\lambda_{1},...,\lambda_{7}$ sont des éléments non nuls de $\mathbb{K}$ (dépendant du choix des $a_{i}$ et des $b_{i}$). Soit $h \in \mathfrak{h}$ tel que  $$\alpha_1(h)=\alpha_4(h)=\alpha_5(h)=-\alpha_2(h)=-\alpha_3(h)=-\alpha_6(h)=-\alpha_7(h)=1.$$
On a $h.f=-f$ et $[h,x]=-x$ de sorte que l'algèbre de Lie $\mathfrak{q}$ est non stable.\\
\\
\newpage
Supposons    $\pi^{\prime}$ de type $D_{5}$,  on a $\pi^{\prime}=\{\alpha_{2},\alpha_{3},\alpha_{4},\alpha_{5},\alpha_{6}\}$. Ici, on a $\mathcal{E}_{\pi^{\prime}}=\{\alpha_{2}+\alpha_{3}+2\alpha_{4}+2\alpha_{5}+\alpha_{6},\alpha_{2}+\alpha_{3}+\alpha_{4},\alpha_{4},\alpha_{6}\}$:\\
\\
\begin{tabular}{|l|l|c|c|c|}
\hline
&$Type$ & $ \pi^{\prime\prime}$ & $\ind(\mathfrak{q}_{\pi^{\prime}, \pi^{\prime\prime}})$ & $\rang(\mathfrak{q}_{\pi^{\prime}, \pi^{\prime\prime}})$ \\
\hline
&$A_{1}^{4} $& $\{\alpha_1,\alpha_2,\alpha_{5},\alpha_{7}\}$ & 1 & 0\\
\hline
&$A_{1}^{3}\times A_{2}$& $\{\alpha_1,\alpha_{2}, \alpha_3,\alpha_{5},\alpha_{7}\}$ & 1 & 0\\
\hline
\end{tabular}\\
\\

Dans ce cas et d'après le tableau ci-dessus, les sous-algèbres biparaboliques\\ $\mathfrak{q_{1}}=\mathfrak{q}_{\{\{\alpha_2,\alpha_{3},\alpha_{4},\alpha_{5},\alpha_{6}\},\{\alpha_{1},\alpha_{2},\alpha_5,\alpha_7\}\}}$ et $\mathfrak{q_{2}}=\mathfrak{q}_{\{\{\alpha_{2},\alpha_{3},\alpha_{4},\alpha_{5},\alpha_6\},\{\alpha_{1},\alpha_{2},\alpha_3,\alpha_{5},\alpha_7\}\}} $\\
sont de rang nul et d'indice strictement positif:\\
Soit $f=K(v,.)\in A$, o\`{u} $v=a_{1}\,x_{-\alpha_{1}}+a_{2}\,x_{-\alpha_{2}}+a_{3}\,x_{-\alpha_5}+a_{4}\,x_{-\alpha_{7}}+b_{1}\,x_{\alpha_{4}}+b_{2}\,
x_{\alpha_{6}}+b_{3}\,x_{\alpha_2+\alpha_3+\alpha_{4}}+b_{4}\,x_{\alpha_2+\alpha_{3}+2\alpha_{4}+2\alpha_{5}+\alpha_6}$ (les $a_{i}$  et les $b_{j}$ sont tous non nuls). Dans ce cas, on a $\mathfrak{q_{1}}(f)=\mathbb{K}x$ avec\\
 $x=x_{\alpha_{5}}+\lambda_{1}\,x_{-\alpha_{4}}+\lambda_{2}\,x_{-\alpha_{6}}+\lambda_{3}\, x_{-\alpha_2-\alpha_3-\alpha_{4}}+\lambda_{4}\, x_{-\alpha_{4}-\alpha_{5}-\alpha_{6}}+\lambda_{5}\,x_{-\alpha_2-\alpha_{3}-2\alpha_4-\alpha_{5}} +\lambda_{6}\,x_{-\alpha_{2}-\alpha_3-\alpha_{4}-\alpha_{5}-\alpha_6}+\lambda_{7}\,x_{-\alpha_2-\alpha_{3}-2\alpha_{4}-2\alpha_{5}-\alpha_{6}},$\\
\\
o\`{u} $\lambda_{1},...,\lambda_{7}$ sont des éléments non nuls de $\mathbb{K}$ (dépendant du choix des $a_{i}$ et des $b_{i}$). Soit $h \in \mathfrak{h}$ tel que  $$\alpha_3(h)=\alpha_4(h)=\alpha_6(h)=-\alpha_1(h)=-\alpha_2(h)=-\alpha_5(h)=-\alpha_7(h)=1.$$
On a $h.f=-f$ et $[h,x]=-x$ de sorte que l'algèbre de Lie $\mathfrak{q_{1}}$ est non stable.\\

Soit $f=K(v,.)\in A$, o\`{u} $v=a_{1}\,x_{-\alpha_{1}-\alpha_3}+a_{2}\,x_{-\alpha_{2}}+a_{3}\,x_{-\alpha_5}+a_{4}\,x_{-\alpha_{7}}+b_{1}\,x_{\alpha_{4}}+b_{2}\,
x_{\alpha_{6}}+b_{3}\,x_{\alpha_2+\alpha_3+\alpha_{4}}+b_{4}\,x_{\alpha_2+\alpha_{3}+2\alpha_{4}+2\alpha_{5}+\alpha_6}$ (les $a_{i}$  et les $b_{j}$ sont tous non nuls). Dans ce cas, on vérifie que  $\mathfrak{q_{2}}(f)= \mathfrak{q_{1}}(f)$. Par suite et d'après ce qui précède, l'algèbre de Lie $\mathfrak{q_{2}}$ est non stable.
\\
\newpage
Supposons   $\pi^{\prime}$ de type $D_{6}$,  on a $\pi^{\prime}=\{\alpha_{2},\alpha_{3},\alpha_{4},\alpha_{5},\alpha_{6},\alpha_{7}\}$. Ici, on a $\mathcal{E}_{\pi^{\prime}}=\{\alpha_{2}+\alpha_{3}+2\alpha_{4}+2\alpha_{5}+2\alpha_{6}+\alpha_{7},\alpha_{2}+\alpha_{3}+2\alpha_{4}+\alpha_{5},
\alpha_{2},\alpha_{3},\alpha_{5},\alpha_{7}\}$:\\
\\
\begin{tabular}{|l|l|l|l|l|l|l|c|c|c|}
\hline
&$Type$ & $ \pi^{\prime\prime}$ & $\ind(\mathfrak{q}_{\pi^{\prime}, \pi^{\prime\prime}})$ & $\rang(\mathfrak{q}_{\pi^{\prime}, \pi^{\prime\prime}})$ \\
\hline
&$A_{1}^{2}$& $\{\alpha_1,\alpha_4\}$ & 1 & 0\\
\hline
&$A_{1}^{2}$& $\{\alpha_1,\alpha_6\}$ & 1 & 0\\
\hline
&$A_{1}\times A_{2}$& $\{\alpha_1,\alpha_2,\alpha_4\}$ & 1 & 1\\
\hline
&$A_{1}\times A_{2}$& $\{\alpha_1,\alpha_4,\alpha_5\}$ & 1 & 1\\
\hline
&$A_{1}\times A_{2}$& $\{\alpha_1,\alpha_5,\alpha_6\}$ & 1 & 1\\
\hline
&$A_{1}\times A_{2}$& $\{\alpha_1,\alpha_6,\alpha_7\}$ & 1 & 1\\
\hline
&$A_{1}\times A_{2}$& $\{\alpha_1,\alpha_3,\alpha_6\}$ & 1 & 0\\
\hline
&$A_{2}\times A_{2}$& $\{\alpha_1,\alpha_3,\alpha_5,\alpha_6\}$ & 1 & 1\\
\hline
&$A_{2}\times A_{2}$& $\{\alpha_1,\alpha_3,\alpha_6,\alpha_7\}$ & 1 & 1\\
\hline
&$A_{4}$& $\{\alpha_1,\alpha_2,\alpha_3,\alpha_4\}$ & 1 & 1\\
\hline
&$A_{4}$& $\{\alpha_1,\alpha_3,\alpha_4,\alpha_5\}$ & 1 & 1\\
\hline
&$A_{1}^{3}$& $\{\alpha_1,\alpha_4,\alpha_6\}$ & 2 & 0\\
\hline
&$A_{1}^{2}\times A_{2}$& $\{\alpha_1,\alpha_{4}, \alpha_6,\alpha_7\}$ & 2 & $1\leq \rang(\mathfrak{q}_{\pi^{\prime}, \pi^{\prime\prime}}) \leq 2$\\
\hline
&$A_{1}^{2}\times A_{2}$&$\{\alpha_1, \alpha_2,\alpha_4,\alpha_6\}$ & 2 & $1\leq \rang(\mathfrak{q}_{\pi^{\prime}, \pi^{\prime\prime}}) \leq 2$\\
\hline
&$A_{1}\times A_{2}^{2}$&$\{\alpha_1, \alpha_2,\alpha_4,\alpha_6,\alpha_7\}$ & 2 & 2\\
\hline
&$A_{1}\times A_{3}$&$\{\alpha_1, \alpha_5,\alpha_6,\alpha_7\}$ & 2 & 2\\
\hline
&$A_{1}\times A_{3}$&$\{\alpha_1, \alpha_2,\alpha_4,\alpha_5\}$ & 2 & 2\\
\hline
&$A_{1}\times A_{4}$&$\{\alpha_1, \alpha_4,\alpha_5,\alpha_6,\alpha_7\}$ & 2& 0\\
\hline
&$A_{1}\times A_{4}$&$\{\alpha_1, \alpha_2,\alpha_4,\alpha_5,\alpha_6\}$ & 2 & 0\\
\hline
&$A_{1}\times A_{4}$&$\{\alpha_1, \alpha_2,\alpha_3,\alpha_4,\alpha_6\}$ & 2 & $1\leq \rang(\mathfrak{q}_{\pi^{\prime}, \pi^{\prime\prime}}) \leq 2$\\
\hline
&$A_{2}\times A_{3}$&$\{\alpha_1, \alpha_3,\alpha_5,\alpha_6,\alpha_7\}$ & 2 & 2\\
\hline
&$A_{2}\times A_{4}$&$\{\alpha_1, \alpha_2,\alpha_3,\alpha_4,\alpha_6,\alpha_7\}$ & 2 & 2\\
\hline
&$A_{5}$&$\{\alpha_1, \alpha_3,\alpha_4,\alpha_5,\alpha_6\}$ & 2 & 2\\
\hline
&$A_{6}$&$\{\alpha_1, \alpha_3,\alpha_4,\alpha_5,\alpha_6,\alpha_7\}$ & 2 & 2\\
\hline
&$E_{6}$&$\{\alpha_1, \alpha_2,\alpha_3,\alpha_4,\alpha_5,\alpha_6\}$ & 3 & 3\\
\hline
\end{tabular}\\
\\

Par exemple et d'après le tableau ci-dessus, $\mathfrak{q}=\mathfrak{q}_{\{\{\alpha_2,\alpha_{3},\alpha_{4},\alpha_{5},\alpha_{6},\alpha_7\},\{\alpha_{1},\alpha_{4},\alpha_{6}\}\}}$ est une sous-algèbre biparabolique de rang nul et d'indice strictement positif:\\
Soit $f=K(v,.)\in A$, o\`{u} $v=a_{1}\,x_{-\alpha_{1}}+a_{2}\,x_{-\alpha_{4}}+a_{3}\,x_{-\alpha_{6}}+b_{1}\,x_{\alpha_{2}}+b_{2}\,
x_{\alpha_{3}}+b_{3}\,x_{\alpha_{5}}+b_{4}\,x_{\alpha_{7}}+b_{5}\,x_{\alpha_{2}+\alpha_{3}+2\alpha_{4}+\alpha_{5}}+b_{6}\,x_{\alpha_{2}+\alpha_{3}+2\alpha_{4}+2\alpha_{5}+2\alpha_{6}+\alpha_{7}}$ (les $a_{i}$  et les $b_{j}$ sont tous non nuls). Dans ce cas, on a $\mathfrak{q}(f)=\mathbb{K}x \oplus \mathbb{K} y$ avec\\
 $x=x_{\alpha_{4}}+\lambda_{1}\,x_{-\alpha_{2}}+\lambda_{2}\,x_{-\alpha_{3}}+\lambda_{3}\, x_{-\alpha_5}+\lambda_{4}\, x_{-\alpha_{2}-\alpha_{3}-\alpha_{4}}+\lambda_{5}\, x_{-\alpha_{2}-\alpha_{4}-\alpha_{5}}+\lambda_{6}\,x_{\alpha_3-\alpha_{4}-\alpha_{5}} +\lambda_{7}\,x_{-\alpha_{2}-\alpha_3-2\alpha_{4}-\alpha_{5}}$\\
et\\
 $y=x_{\alpha_{6}}+\lambda^{\prime}_{1}\,x_{-\alpha_{5}}+\lambda^{\prime}_{2}\,x_{-\alpha_{7}}+\lambda^{\prime}_{3}\, x_{-\alpha_5-\alpha_6-\alpha_7}+\lambda^{\prime}_{4}\, x_{-\alpha_{2}-\alpha_{3}-2\alpha_{4}-\alpha_5}+\lambda^{\prime}_{5}\, x_{-\alpha_{2}-\alpha_3-2\alpha_4-2\alpha_{5}-\alpha_{6}}+\lambda^{\prime}_{6}\,x_{-\alpha_{2}-\alpha_3-2\alpha_4-\alpha_{5}-\alpha_{6}-\alpha_{7}} +\lambda^{\prime}_{7}\,x_{-\alpha_{2}-\alpha_3-2\alpha_4-2\alpha_{5}-2\alpha_{6}-\alpha_{7}}$\\
\\
o\`{u} les $\lambda_{i}$ et $\lambda^{\prime}_{i}$ sont des éléments non nuls de $\mathbb{K}$ (dépendant du choix des $a_{i}$ et des $b_{i}$). Soit $h \in \mathfrak{h}$ tel que  $\alpha_2(h)=\alpha_3(h)=\alpha_5(h)=\alpha_7(h)=-\alpha_1(h)=-\alpha_4(h)=-\alpha_6(h)=1$. On a $h.f=-f$, $[h,x]=-x$ et $[h,y]=-y$ de sorte que l'algèbre de Lie $\mathfrak{q}$ est non stable.\\

On procède de la même manière pour les autres sous-algèbres biparaboliques de rang nul  et d'indice strictement positif.

Supposons  $\pi^{\prime}$ de type $E_{6}$, on a  $\pi^{\prime}=\{\alpha_1,\alpha_{2},\alpha_{3},\alpha_{4},\alpha_{5},\alpha_{6}\}$. Ici, on a $\mathcal{E}_{\pi^{\prime}}=\{\alpha_{1}+2\alpha_{2}+2\alpha_{3}+3\alpha_{4}+2\alpha_{5}+\alpha_{6},\alpha_{1}+\alpha_{3}+\alpha_{4}+\alpha_{5}+\alpha_6,
\alpha_{3}+\alpha_{4}+\alpha_{5},\alpha_{4}\}$:\\
\\
\begin{tabular}{|l|l|l|c|c|c|}
\hline
&$Type$ & $ \pi^{\prime\prime}$ & $\ind(\mathfrak{q}_{\pi^{\prime}, \pi^{\prime\prime}})$ & $\rang(\mathfrak{q}_{\pi^{\prime}, \pi^{\prime\prime}})$ \\
\hline
&$A_{1}^{4}$& $\{\alpha_1,\alpha_2,\alpha_5,\alpha_7\}$ & 1 & 0\\
\hline
&$A_{1}^{3}\times A_{2}$&$\{\alpha_1,\alpha_2, \alpha_3,\alpha_5,\alpha_7\}$ & 1 & 0\\
\hline
&$A_{1}^{2}\times A_{3}$&$\{\alpha_2, \alpha_3,\alpha_5,\alpha_6,\alpha_7\}$ & 1 & 0\\
\hline
&$A_{1}\times  A_{2}\times A_{3}$&$\{\alpha_1, \alpha_2,\alpha_3,\alpha_5,\alpha_6,\alpha_7\}$ & 1 & 0\\
\hline
&$A_{1}\times A_{5}$&$\{\alpha_1, \alpha_2,\alpha_4,\alpha_5,\alpha_6,\alpha_7\}$ & 1 & 1\\
\hline
\end{tabular}\\
\\

Par exemple et d'après le tableau ci-dessus, $\mathfrak{q}=\mathfrak{q}_{\{\{\alpha_1,\alpha_{2},\alpha_{3},\alpha_{4},\alpha_{5},\alpha_{6}\},\{\alpha_{2},\alpha_{3},\alpha_5,\alpha_{6},\alpha_7\}\}}$ est une sous-algèbre biparabolique de rang nul et d'indice strictement positif:\\
Soit $f=K(v,.)\in A$, o\`{u} $v=a_{1}\,x_{-\alpha_{1}}+a_{2}\,x_{-\alpha_{2}}+a_{3}\,x_{-\alpha_5}+a_{4}\,x_{-\alpha_{7}}+b_{1}\,x_{\alpha_{4}}+b_{2}\,
x_{\alpha_3+\alpha_{4}+\alpha_5}+b_{3}\,x_{\alpha_1+\alpha_3+\alpha_4+\alpha_5+\alpha_{6}}+b_{4}\,x_{\alpha_1+2\alpha_{2}+2\alpha_{3}+3\alpha_{4}+2\alpha_{5}+\alpha_6}$ (les $a_{i}$  et les $b_{j}$ sont tous non nuls). Dans ce cas, on a $\mathfrak{q}(f)=\mathbb{K}x$ avec\\
 $x=x_{\alpha_{2}}+\lambda_{1}\,x_{-\alpha_{4}}+\lambda_{2}\,x_{-\alpha_3-\alpha_{4}-\alpha_5}+\lambda_{3}\, x_{-\alpha_1-\alpha_3-\alpha_{4}-\alpha_5-\alpha_6}+\lambda_{4}\, x_{-\alpha_{2}-\alpha_3-2\alpha_{4}-\alpha_{5}}+\lambda_{5}\,x_{-\alpha_1-\alpha_{2}-\alpha_3-2\alpha_{4}-\alpha_{5}-\alpha_6} +\lambda_{6}\,x_{-\alpha_1-\alpha_{2}-2\alpha_3-2\alpha_{4}-2\alpha_{5}-\alpha_6}+\lambda_{7}\,x_{-\alpha_1-2\alpha_{2}-2\alpha_{3}-3\alpha_{4}-2\alpha_{5}}-\alpha_6,$\\
\\
o\`{u} $\lambda_{1},...,\lambda_{7}$ sont des éléments non nuls de $\mathbb{K}$ (dépendant du choix des $a_{i}$ et des $b_{i}$). Soit $h \in \mathfrak{h}$ tel que  $$\alpha_3(h)=\alpha_4(h)=\alpha_6(h)=-\alpha_1(h)=-\alpha_2(h)=-\alpha_5(h)=-\alpha_7(h)=1.$$
 On a $h.f=-f$ et $[h,x]=-x$ de sorte que l'algèbre de Lie $\mathfrak{q}$ est non stable.\\

On procède de la même manière pour les autres sous-algèbres biparaboliques de rang nul  et d'indice strictement positif.\\
\\
\\
\textbf{Cas o\`{u} $\pi^{\prime}$ est non connexe et $\pi^{\prime\prime}$ est non connexe ou tel que $\mathcal{K}(\pi^{\prime\prime})$ est de cardinal inférieur ou égal à 3: }\\

Dans ce cas, il est clair que si la sous-algèbre biparabolique $\mathfrak{q}_{\pi^{\prime}, \pi^{\prime\prime}}$ vérifie la condition $(\ast)$ alors $\pi^{\prime}$ ou $\pi^{\prime\prime}$ est de type $A_{1}^{2}\times A_{1}^{2}$, $A_{1}^{2}\times A_{1}\times A_{2}$, $A_{1}^{2}\times A_{3}$, $A_{1}\times A_{2}\times A_{3}$, $A_{1}\times A_{5}$, $D_{4}\times A_{1}$ ou  $D_{5}\times A_{1}$ .\\

Supposons   $\pi^{\prime}$ de type $A_{1}^{2}\times A_{3}$,   $\pi^{\prime}=\{\alpha_1,\alpha_{2},\alpha_{4},\alpha_{5},\alpha_{7}\}$. Ici, on a $\mathcal{E}_{\pi^{\prime}}=\{\alpha_{2}+\alpha_{4}+\alpha_{5},\alpha_{1},\alpha_{4},\alpha_7\}$:\\
\\
\begin{tabular}{|l|l|c|c|c|}
\hline
&$Type$ & $ \pi^{\prime\prime}$ & $\ind(\mathfrak{q}_{\pi^{\prime}, \pi^{\prime\prime}})$ & $\rang(\mathfrak{q}_{\pi^{\prime}, \pi^{\prime\prime}})$ \\
\hline
&$A_{1}^{2}\times A_{3} $& $\{\alpha_2,\alpha_3,\alpha_5,\alpha_6,\alpha_7\}$ & 1 & 1\\
\hline
&$A_{1}\times A_{2}\times A_{3}$& $\{\alpha_1,\alpha_2, \alpha_3,\alpha_5,\alpha_6,\alpha_7\}$ & 1 & 1\\
\hline
\end{tabular}\\
\\

Supposons   $\pi^{\prime}$ de type $A_{1}\times A_{5}$,   $\pi^{\prime}=\{\alpha_1,\alpha_{2},\alpha_{4},\alpha_{5},\alpha_6,\alpha_{7}\}$. Ici, on a $\mathcal{E}_{\pi^{\prime}}=\{\alpha_{2}+\alpha_{4}+\alpha_{5}+\alpha_{6}+\alpha_{7},\alpha_{4}+\alpha_{5}+\alpha_{6},\alpha_{5},\alpha_{1}\}$:\\
\\
\begin{tabular}{|l|l|c|c|c|}
\hline
&$Type$ & $ \pi^{\prime\prime}$ & $\ind(\mathfrak{q}_{\pi^{\prime}, \pi^{\prime\prime}})$ & $\rang(\mathfrak{q}_{\pi^{\prime}, \pi^{\prime\prime}})$ \\
\hline
&$A_{1}^{2}\times A_{3} $& $\{\alpha_2,\alpha_3,\alpha_5,\alpha_6,\alpha_7\}$ & 1 & 1\\
\hline
&$A_{1}\times A_{2}\times A_{3}$& $\{\alpha_1,\alpha_2, \alpha_3,\alpha_5,\alpha_6,\alpha_7\}$ & 1 & 1\\
\hline
\end{tabular}\\
\\

Supposons   $\pi^{\prime}$ de type $D_{4}\times A_{1}$,   $\pi^{\prime}=\{\alpha_2,\alpha_{3},\alpha_{4},\alpha_{5},\alpha_{7}\}$. Ici, on a $\mathcal{E}_{\pi^{\prime}}=\{\alpha_{2}+\alpha_{3}+2\alpha_{4}+\alpha_{5},\alpha_{2},\alpha_{3},\alpha_{5},\alpha_7\}$:\\
\\
\begin{tabular}{|l|l|l|l|l|l|c|c|c|}
\hline
&$Type$ & $ \pi^{\prime\prime}$ & $\ind(\mathfrak{q}_{\pi^{\prime}, \pi^{\prime\prime}})$ & $\rang(\mathfrak{q}_{\pi^{\prime}, \pi^{\prime\prime}})$ \\
\hline
&$A_{1}^{3}$& $\{\alpha_1,\alpha_4,\alpha_6\}$ & 1 & 0\\
\hline
&$A_{1}^{2}\times A_{2}$& $\{\alpha_1,\alpha_2, \alpha_4,\alpha_6\}$ & 1 & 1\\
\hline
&$A_{1}^{2}\times A_{2}$&$\{\alpha_1,\alpha_4,\alpha_6,\alpha_7\}$ & 1 & 0\\
\hline
&$A_{1}\times A_{2}^{2}$&$\{\alpha_1, \alpha_2,\alpha_4,\alpha_6,\alpha_7\}$ & 1 & 1\\
\hline
&$ A_{1}\times A_{3}$&$\{\alpha_1,\alpha_5,\alpha_6,\alpha_7\}$ & 1 & 1\\
\hline
&$A_{2}\times A_{3}$&$\{\alpha_1,\alpha_3,\alpha_5,\alpha_6,\alpha_7\}$ & 1 & 1\\
\hline
&$A_{1}\times A_{4}$&$\{\alpha_1, \alpha_2,\alpha_4,\alpha_5,\alpha_6\}$ & 1 & 1\\
\hline
&$A_{1}\times A_{4}$&$\{\alpha_1, \alpha_2,\alpha_3,\alpha_4,\alpha_6\}$ & 1 & 1\\
\hline
&$A_{1}\times A_{4}$&$\{\alpha_1, \alpha_4,\alpha_5,\alpha_6,\alpha_7\}$ & 1 & 0\\
\hline
&$A_{2}\times A_{4}$&$\{\alpha_1, \alpha_2,\alpha_3,\alpha_4,\alpha_6,\alpha_7\}$ & 1 & 1\\
\hline
&$A_{6}$&$\{\alpha_1, \alpha_3,\alpha_4,\alpha_5,\alpha_6,\alpha_7\}$ & 1 & 1\\
\hline
\end{tabular}\\
\\

Supposons   $\pi^{\prime}$ de type $D_{5}\times A_{1}$,   $\pi^{\prime}=\{\alpha_1,\alpha_2,\alpha_{3},\alpha_{4},\alpha_{5},\alpha_{7}\}$. Ici, on a $\mathcal{E}_{\pi^{\prime}}=\{\alpha_1+\alpha_{2}+2\alpha_{3}+2\alpha_{4}+\alpha_{5},\alpha_{2}+\alpha_{4}+\alpha_{5},\alpha_1,\alpha_4,\alpha_7\}$:\\
\\
\begin{tabular}{|l|l|l|l|l|l|l|c|c|c|}
\hline
&$Type$ & $ \pi^{\prime\prime}$ & $\ind(\mathfrak{q}_{\pi^{\prime}, \pi^{\prime\prime}})$ & $\rang(\mathfrak{q}_{\pi^{\prime}, \pi^{\prime\prime}})$ \\
\hline
&$A_{1}^{3}$& $\{\alpha_2,\alpha_3,\alpha_6\}$ & 1 & 0\\
\hline
&$A_{1}^{2}\times A_{2}$& $\{\alpha_2,\alpha_3, \alpha_5,\alpha_6\}$ & 1 & 0\\
\hline
&$A_{1}^{2}\times A_{2}$&$\{\alpha_2,\alpha_3,\alpha_6,\alpha_7\}$ & 1 & 0\\
\hline
&$A_{1}\times A_{2}^{2}$&$\{\alpha_1, \alpha_2,\alpha_3,\alpha_6,\alpha_7\}$ & 1 & 1\\
\hline
&$A_{1}\times A_{2}^{2} $&$\{\alpha_1, \alpha_2,\alpha_3,\alpha_5,\alpha_6\}$ & 1 & 1\\
\hline
&$A_{1}\times A_{3}$&$\{\alpha_2,\alpha_5,\alpha_6,\alpha_7\}$ & 1 & 1\\
\hline
&$A_{1}\times A_{3}$&$\{\alpha_3,\alpha_5,\alpha_6,\alpha_7\}$ & 1 & 0\\
\hline
&$A_{2}\times A_{3}$&$\{\alpha_1, \alpha_3,\alpha_5,\alpha_6,\alpha_7\}$ & 1 & 1\\
\hline
&$A_{1}\times A_{4}$&$\{\alpha_1, \alpha_2,\alpha_3,\alpha_4,\alpha_6\}$ & 1 & 1\\
\hline
&$A_{2}\times A_{4}$&$\{\alpha_1, \alpha_2,\alpha_3,\alpha_4,\alpha_6,\alpha_7\}$ & 1 & 1\\
\hline
&$ A_{5}$&$\{\alpha_2, \alpha_4,\alpha_5,\alpha_6,\alpha_7\}$ & 1 & 1\\
\hline
&$ A_{5}$&$\{\alpha_3, \alpha_4,\alpha_5,\alpha_6,\alpha_7\}$ & 1 & 0\\
\hline
&$ A_{6}$&$\{\alpha_1, \alpha_3,\alpha_4,\alpha_5,\alpha_6,\alpha_7\}$ & 1 & 1\\
\hline
&$A_{1}^{2}\times A_{3} $&$\{\alpha_2, \alpha_3,\alpha_5,\alpha_6,\alpha_7\}$ & 2& $1\leq \rang(\mathfrak{q}_{\pi^{\prime}, \pi^{\prime\prime}}) \leq 2$\\
\hline
&$A_{1}\times A_{2}\times A_{3} $&$\{\alpha_1,\alpha_2, \alpha_3,\alpha_5,\alpha_6,\alpha_7\}$ & 2 & 2\\
\hline
\end{tabular}\\
\\
\\
\\
\textbf{Calculs dans $E_{8}$ }:\\

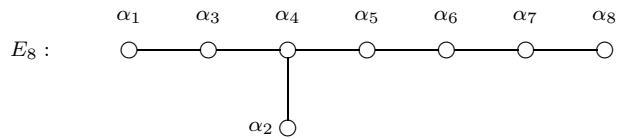
$$\begin{Dynkin}
\Dbloc{\Dtext{l}{E_8:}}
\Dbloc{\Dcirc \Deast\Dtext{t}{\alpha_{1}} }
 \Dbloc{\Dwest\Dcirc \Deast\Dtext{t}{\alpha_{3}} }
 \Dbloc{\Dwest \Dcirc \Dsouth \Deast\Dtext{t}{\alpha_{4}} }
 \Dbloc{\Dwest \Dcirc \Deast\Dtext{t}{\alpha_{5}}}
\Dbloc{ \Dwest  \Deast \Dcirc\Dtext{t}{\alpha_{6}} }
\Dbloc{ \Dwest \Deast \Dcirc\Dtext{t}{\alpha_{7}} }
\Dbloc{ \Dwest \Dcirc\Dtext{t}{\alpha_{8}} }
\Dskip
\Dspace \Dspace \Dspace \Dbloc{\Dnorth \Dcirc\Dtext{l}{\alpha_{2}}}
\end{Dynkin}$$
\\

Remarquons tout d'abord que si la sous-algèbre biparabolique $\mathfrak{q}_{\pi^{\prime}, \pi^{\prime\prime}}$ vérifie la condition $(\ast)$ alors $\pi^{\prime}$ est tel que $\mathcal{K}(\pi^{\prime})$ est de cardinal supérieur ou égal à 5. \\
\\
\\
\textbf{Cas o\`{u} $\pi^{\prime}$ est connexe: }
\\

Supposons  que  $\pi^{\prime}$ de type $D_{6}$,   $\pi^{\prime}=\{\alpha_{2},\alpha_{3},\alpha_{4},\alpha_{5},\alpha_{6},\alpha_{7}\}$. Ici, on a $\mathcal{E}_{\pi^{\prime}}=\{\alpha_{2}+\alpha_{3}+2\alpha_{4}+2\alpha_{5}+2\alpha_{6}+\alpha_{7},\alpha_{2}+\alpha_{3}+2\alpha_{4}+\alpha_{5},\alpha_{7},
\alpha_{2},\alpha_{3},\alpha_{5}\}$.\\

On commence tout d'abord par le cas o\`{u} le cardinal de $\mathcal{K}(\pi^{\prime\prime})$ est égal à $3$:\\
\\
\begin{tabular}{|l|l|l|l|l|l|l|c|c|c|}
\hline
&$Type$ & $ \pi^{\prime\prime}$ & $\ind(\mathfrak{q}_{\pi^{\prime}, \pi^{\prime\prime}})$ & $\rang(\mathfrak{q}_{\pi^{\prime}, \pi^{\prime\prime}})$ \\
\hline
&$A_{1}^{3}$& $\{\alpha_1,\alpha_4,\alpha_8\}$ &  1& 0\\
\hline
&$A_{1}^{3}$& $\{\alpha_1,\alpha_{6}, \alpha_8\}$ & 1 &0 \\
\hline
&$A_{1}^{2}\times A_{2}$&$\{\alpha_1, \alpha_4,\alpha_7,\alpha_8\}$ & 1 & 0\\
\hline
&$A_{1}^{2}\times A_{2}$&$\{\alpha_1, \alpha_3,\alpha_6,\alpha_8\}$ & 1 & 0\\
\hline
&$A_{1}\times A_{2}^{2}$&$\{\alpha_1, \alpha_2,\alpha_4,\alpha_7,\alpha_8\}$ & 1 & 1\\
\hline
&$A_{1}\times A_{2}^{2}$&$\{\alpha_1, \alpha_4,\alpha_5,\alpha_7,\alpha_8\}$ & 1 & 1\\
\hline
&$A_{1}\times A_{2}^{2}$&$\{\alpha_1, \alpha_3,\alpha_5,\alpha_6,\alpha_8\}$ &  1&1 \\
\hline
&$A_{1}\times A_{4}$&$\{\alpha_1, \alpha_5,\alpha_6,\alpha_7,\alpha_8\}$ & 1 &1 \\
\hline
&$A_{1}\times A_{4}$&$\{\alpha_1, \alpha_2,\alpha_3,\alpha_4,\alpha_8\}$ & 1 & 1\\
\hline
&$A_{1}\times A_{4}$&$\{\alpha_1, \alpha_3,\alpha_4,\alpha_5,\alpha_8\}$ & 1& 1\\
\hline
&$A_{2}\times A_{4}$&$\{\alpha_1, \alpha_3,\alpha_5,\alpha_6,\alpha_7,\alpha_8\}$ & 1 &1 \\
\hline
&$A_{2}\times A_{4}$&$\{\alpha_1, \alpha_2,\alpha_3,\alpha_4,\alpha_7,\alpha_8\}$ & 1 & 1\\
\hline
&$A_{2}\times A_{4}$&$\{\alpha_1, \alpha_3,\alpha_4,\alpha_5,\alpha_7,\alpha_8\}$ & 1 & 1\\
\hline
\end{tabular}\\
\\

Cas o\`{u} le cardinal de $\mathcal{K}(\pi^{\prime\prime})$ est égal à 4:\\
\\
\begin{tabular}{|l|l|l|l|c|c|c|}
\hline
&$Type$ & $ \pi^{\prime\prime}$ & $\ind(\mathfrak{q}_{\pi^{\prime}, \pi^{\prime\prime}})$ & $\rang(\mathfrak{q}_{\pi^{\prime}, \pi^{\prime\prime}})$ \\
\hline
&$A_{1}^{4}$& $\{\alpha_1,\alpha_4,\alpha_6,\alpha_8\}$ & 2 & 0\\
\hline
&$A_{1}^{3} \times A_{2}$& $\{\alpha_1,\alpha_2,\alpha_4,\alpha_6,\alpha_8\}$ &2  &$1\leq \rang(\mathfrak{q}_{\pi^{\prime}, \pi^{\prime\prime}}) \leq 2 $\\
\hline
&$A_{1}^{2}\times A_{3}$&$\{\alpha_1, \alpha_2,\alpha_4,\alpha_5,\alpha_8\}$ &2  & 2\\
\hline
&$A_{1}^{2}\times A_{4}$&$\{\alpha_1, \alpha_2,\alpha_4,\alpha_5,\alpha_6,\alpha_8\}$ & 2 & 0\\
\hline
&$A_{1}^{2}\times A_{4}$&$\{\alpha_1, \alpha_2,\alpha_3,\alpha_4,\alpha_6,\alpha_8\}$ &  2& 1$\leq \rang(\mathfrak{q}_{\pi^{\prime}, \pi^{\prime\prime}}) \leq 2$\\
\hline
&$A_{1}\times A_{2}\times A_{3}$&$\{\alpha_1, \alpha_2,\alpha_4,\alpha_5,\alpha_7,\alpha_8\}$ & 2 & 2\\
\hline
&$A_{1}\times A_{5}$&$\{\alpha_1, \alpha_3,\alpha_4, \alpha_5,\alpha_6,\alpha_8\}$ &2  & 2\\
\hline
&$A_{1}\times A_{5}$&$\{\alpha_1, \alpha_4,\alpha_5,\alpha_6,\alpha_7,\alpha_8\}$ &2  & 2\\
\hline
&$A_{1}\times A_{6}$&$\{\alpha_1, \alpha_2,\alpha_4,\alpha_5,\alpha_6,\alpha_7,\alpha_8\}$ &2  &2 \\
\hline
\end{tabular}\\
\\

Cas o\`{u} le cardinal de $\mathcal{K}(\pi^{\prime\prime})$ est égal à 5:\\
\\
\begin{tabular}{|l|l|l|l|c|c|c|}
\hline
&$Type$ & $ \pi^{\prime\prime}$ & $\ind(\mathfrak{q}_{\pi^{\prime}, \pi^{\prime\prime}})$ & $\rang(\mathfrak{q}_{\pi^{\prime}, \pi^{\prime\prime}})$ \\
\hline
&$D_{5}\times A_{1}$& $\{\alpha_1,\alpha_2,\alpha_3,\alpha_4,\alpha_5,\alpha_8\}$ &3  & 3\\
\hline
&$E_{6}\times A_{1}$& $\{\alpha_1,\alpha_2,\alpha_3,\alpha_4,\alpha_5,\alpha_6,\alpha_8\}$ & 3 &3 \\
\hline
&$D_{5}\times A_{2}$& $\{\alpha_1, \alpha_2,\alpha_3,\alpha_4,\alpha_5,\alpha_7,\alpha_8\}$ & 3 & 3\\
\hline
\end{tabular}\\
\\

Supposons  maintenant  $\pi^{\prime}$ de type $D_{7}$, on a $\pi^{\prime}=\{\alpha_{2},\alpha_{3},\alpha_{4},\alpha_{5},\alpha_{6},\alpha_{7},\alpha_{8}\}$. Ici, on a $\mathcal{E}_{\pi^{\prime}}=\{\alpha_{2}+\alpha_{3}+2\alpha_{4}+2\alpha_{5}+2\alpha_{6}+2\alpha_{7}+\alpha_{8},\alpha_{2}+\alpha_{3}+2\alpha_{4}+2\alpha_{5}+
\alpha_{6},\alpha_{2}+\alpha_{3}+\alpha_{4},\alpha_{6},\alpha_{4},\alpha_{8}\}$.\\

On commence tout d'abord par le cas o\`{u} le cardinal de $\mathcal{K}(\pi^{\prime\prime})$ est égal à $3$:\\
\\
\begin{tabular}{|l|l|l|l|l|l|l|c|c|c|}
\hline
&$Type$ & $ \pi^{\prime\prime}$ & $\ind(\mathfrak{q}_{\pi^{\prime}, \pi^{\prime\prime}})$ & $\rang(\mathfrak{q}_{\pi^{\prime}, \pi^{\prime\prime}})$ \\
\hline
&$A_{1}^{3}$& $\{\alpha_1,\alpha_2,\alpha_5\}$ & 1 & 0\\
\hline
&$A_{1}^{3}$& $\{\alpha_1,\alpha_{2}, \alpha_7\}$ & 1 & 0\\
\hline
&$A_{1}^{2}\times A_{2}$&$\{\alpha_1, \alpha_2,\alpha_5,\alpha_6\}$ & 1 &1 \\
\hline
&$A_{1}^{2}\times A_{2}$&$\{\alpha_1, \alpha_2,\alpha_6,\alpha_7\}$ & 1 & 1\\
\hline
&$A_{1}^{2}\times A_{2}$&$\{\alpha_1, \alpha_2,\alpha_7,\alpha_8\}$ & 1 & 1\\
\hline
&$A_{1}^{2}\times A_{2}$&$\{\alpha_1, \alpha_2,\alpha_3,\alpha_5\}$ & 1 & 0\\
\hline
&$A_{1}^{2}\times A_{2}$&$\{\alpha_1, \alpha_2,\alpha_3,\alpha_7\}$ & 1 & 0\\
\hline
&$A_{1}\times A_{2}^{2}$&$\{\alpha_1, \alpha_2,\alpha_4,\alpha_6,\alpha_7\}$ & 1 & 1\\
\hline
&$A_{1}\times A_{2}^{2}$&$\{\alpha_1, \alpha_2,\alpha_4,\alpha_7,\alpha_8\}$ & 1 & 1\\
\hline
&$A_{1}\times A_{2}^{2}$&$\{\alpha_1, \alpha_2,\alpha_3,\alpha_5,\alpha_6\}$ & 1 & 1\\
\hline
&$A_{1}\times A_{2}^{2}$&$\{\alpha_1, \alpha_2,\alpha_3,\alpha_6,\alpha_7\}$ &1  & 1\\
\hline
&$A_{1}\times A_{2}^{2}$&$\{\alpha_1, \alpha_2,\alpha_3,\alpha_7,\alpha_8\}$ & 1 & 1\\
\hline
&$A_{1}\times A_{3}$&$\{\alpha_1, \alpha_3,\alpha_4,\alpha_7\}$ & 1 & 0\\
\hline
&$A_{2}\times A_{3}$&$\{\alpha_1, \alpha_3,\alpha_4,\alpha_6,\alpha_7\}$ & 1 & 1\\
\hline
&$A_{2}\times A_{3}$&$\{\alpha_1, \alpha_3,\alpha_4,\alpha_7,\alpha_8\}$ & 1& 1\\
\hline
&$A_{1}\times A_{4}$&$\{\alpha_1, \alpha_2,\alpha_4,\alpha_5,\alpha_6\}$ & 1 &1 \\
\hline
&$A_{1}\times A_{4}$&$\{\alpha_1, \alpha_2,\alpha_3,\alpha_4,\alpha_7\}$ & 1 & 0\\
\hline
&$A_{1}\times A_{4}$&$\{\alpha_1, \alpha_3,\alpha_4,\alpha_5,\alpha_7\}$ & 1 & 0\\
\hline
&$A_{2}\times A_{4}$&$\{\alpha_1, \alpha_2,\alpha_3,\alpha_4,\alpha_6,\alpha_7\}$ & 1 & 1\\
\hline
&$A_{2}\times A_{4}$&$\{\alpha_1, \alpha_2,\alpha_3,\alpha_4,\alpha_7,\alpha_8\}$ &  1&1 \\
\hline
&$A_{2}\times A_{4}$&$\{\alpha_1, \alpha_3,\alpha_4,\alpha_5,\alpha_7,\alpha_8\}$ &  1&1 \\
\hline
&$A_{6}$&$\{\alpha_1, \alpha_3,\alpha_4,\alpha_5,\alpha_6,\alpha_7\}$ & 1 & 1\\
\hline
\end{tabular}\\
\\
\newpage
Cas o\`{u} le cardinal de $\mathcal{K}(\pi^{\prime\prime})$ est supérieur ou égal à 4:\\
\\
\begin{tabular}{|l|l|l|l|l|l|l|c|c|c|}
\hline
&$Type$ & $ \pi^{\prime\prime}$ & $\ind(\mathfrak{q}_{\pi^{\prime}, \pi^{\prime\prime}})$ & $\rang(\mathfrak{q}_{\pi^{\prime}, \pi^{\prime\prime}})$ \\
\hline
&$A_{1}^{4}$& $\{\alpha_1,\alpha_2,\alpha_5, \alpha_7\}$ &2  & 0\\
\hline
&$A_{1}^{3}\times A_{2}$& $\{\alpha_1,\alpha_{2}, \alpha_5,\alpha_7,\alpha_8\}$ & 2 & $1\leq \rang(\mathfrak{q}_{\pi^{\prime}, \pi^{\prime\prime}}) \leq 2$\\
\hline
&$A_{1}^{3}\times A_{2}$&$\{\alpha_1,\alpha_{3}, \alpha_2,\alpha_5,\alpha_7\}$ & 2 & 0\\
\hline
&$A_{1}^{2}\times A_{2}^{2}$&$\{\alpha_1,\alpha_{2}, \alpha_3,\alpha_5,\alpha_7,\alpha_8\}$ &2  & $1\leq \rang(\mathfrak{q}_{\pi^{\prime}, \pi^{\prime\prime}}) \leq 2$\\
\hline
&$A_{1}^{2}\times A_{3}$&$\{\alpha_1, \alpha_2,\alpha_6,\alpha_7,\alpha_8\}$ & 2 & 2\\
\hline
&$A_{1}\times A_{2}\times A_{3}$&$\{\alpha_1, \alpha_2,\alpha_4,\alpha_6,\alpha_7,\alpha_8\}$ & 2 & 2\\
\hline
&$A_{1}\times A_{2}\times A_{3}$&$\{\alpha_1, \alpha_2,\alpha_3,\alpha_6,\alpha_7,\alpha_8\}$ &  2& 2\\
\hline
&$A_{1}^{2}\times A_{4}$&$\{\alpha_1, \alpha_2,\alpha_5,\alpha_6,\alpha_7,\alpha_8\}$ & 2 &2 \\
\hline
&$A_{1}\times A_{2}\times A_{4}$&$\{\alpha_1, \alpha_2, \alpha_3,\alpha_5,\alpha_6,\alpha_7,\alpha_8\}$ &2  & 2\\
\hline
&$A_{1}\times A_{5}$&$\{\alpha_1, \alpha_2,\alpha_4,\alpha_5,\alpha_6,\alpha_7\}$ & 2 & 2\\
\hline
&$A_{1}\times A_{6}$&$\{\alpha_1, \alpha_2,\alpha_4,\alpha_5,\alpha_6,\alpha_7,\alpha_8\}$ & 2 &2 \\
\hline
&$A_{3}\times A_{3}$&$\{\alpha_1, \alpha_3,\alpha_4,\alpha_6,\alpha_7,\alpha_8\}$ & 2 & 2\\
\hline
&$A_{3}\times A_{4}$&$\{\alpha_1,\alpha_2, \alpha_3,\alpha_4,\alpha_6,\alpha_7,\alpha_8\}$ &2  &2 \\
\hline
&$A_{7}$&$\{\alpha_1,\alpha_3,\alpha_4, \alpha_5,\alpha_6,\alpha_7,\alpha_8\}$ & 2 & 2\\
\hline
&$E_{7}$&$\{\alpha_1,\alpha_2,\alpha_3, \alpha_4,\alpha_5,\alpha_6,\alpha_7\}$ & 5 & 5\\
\hline

\end{tabular}\\
\\

Supposons  maintenant   $\pi^{\prime}$ de type $E_{7}$, on a  $\pi^{\prime}=\{\alpha_{1},\alpha_{2},\alpha_{3},\alpha_{4},\alpha_{5},\alpha_{6},\alpha_{7}\}$. Ici, on a $\mathcal{E}_{\pi^{\prime}}=\{2\alpha_1+2\alpha_{2}+3\alpha_{3}+4\alpha_{4}+3\alpha_{5}+2\alpha_{6}+\alpha_{7},\alpha_{2}+\alpha_{3}+2\alpha_{4}+2\alpha_{5}+
2\alpha_{6}+\alpha_7,\alpha_{2}+\alpha_{3}+2\alpha_{4}+\alpha_5,\alpha_{7},\alpha_{2},\alpha_{3},\alpha_5\}$.\\

On commence tout d'abord par le cas o\`{u} le cardinal de $\mathcal{K}(\pi^{\prime\prime})$ est égal à $2$:\\
\\
\\
\begin{tabular}{|l|l|l|l|c|c|c|}
\hline
&$Type$ & $ \pi^{\prime\prime}$ & $\ind(\mathfrak{q}_{\pi^{\prime}, \pi^{\prime\prime}})$ & $\rang(\mathfrak{q}_{\pi^{\prime}, \pi^{\prime\prime}})$ \\
\hline
&$A_{1}^{2}$& $\{\alpha_1,\alpha_8\}$ &1  & 0\\
\hline
&$A_{1}^{2}$& $\{\alpha_4,\alpha_8\}$ & 1 & 0\\
\hline
&$A_{1}^{2}$& $\{\alpha_6,\alpha_8\}$ & 1 &0 \\
\hline
&$A_{1}\times A_{2}$& $\{\alpha_1,\alpha_3,\alpha_8\}$ &  1&1 \\
\hline
&$A_{1}\times A_{2}$& $\{\alpha_3,\alpha_4,\alpha_8\}$ & 1 & 1\\
\hline
&$A_{1}\times A_{2}$& $\{\alpha_2, \alpha_4,\alpha_8\}$ &1  & 1\\
\hline
&$A_{1}\times A_{2}$& $\{\alpha_4,\alpha_5,\alpha_8\}$ & 1 & 1\\
\hline
&$A_{1}\times A_{2}$& $\{\alpha_5,\alpha_6,\alpha_8\}$ & 1 & 1\\
\hline
&$A_{1}\times A_{2}$& $\{\alpha_1,\alpha_7,\alpha_8\}$ & 1 & 0\\
\hline
&$A_{1}\times A_{2}$& $\{\alpha_4,\alpha_7,\alpha_8\}$ & 1 & 0\\
\hline
&$A_{2}\times A_{2}$& $\{\alpha_1,\alpha_3,\alpha_7,\alpha_8\}$ & 1 & 1\\
\hline
&$A_{2}\times A_{2}$& $\{\alpha_3,\alpha_4,\alpha_7,\alpha_8\}$ & 1 & 1\\
\hline
&$A_{2}\times A_{2}$& $\{ \alpha_2,\alpha_4,\alpha_7,\alpha_8\}$ & 1 & 1\\
\hline
&$A_{2}\times A_{2}$& $\{\alpha_4,\alpha_5,\alpha_7,\alpha_8\}$ & 1 & 1\\
\hline
&$A_{4}$& $\{\alpha_5,\alpha_6,\alpha_7,\alpha_8\}$ & 1 & 1\\
\hline
\end{tabular}\\
\\
\newpage
Cas o\`{u} le cardinal de $\mathcal{K}(\pi^{\prime\prime})$ est  égal à 3:\\
\\
\begin{tabular}{|l|l|l|l|l|l|l|c|c|c|}
\hline
&$Type$ & $ \pi^{\prime\prime}$ & $\ind(\mathfrak{q}_{\pi^{\prime}, \pi^{\prime\prime}})$ & $\rang(\mathfrak{q}_{\pi^{\prime}, \pi^{\prime\prime}})$ \\
\hline
&$A_{1}^{3}$& $\{\alpha_1,\alpha_4,\alpha_8\}$ & 2 & 0\\
\hline
&$A_{1}^{3}$& $\{\alpha_1,\alpha_{6}, \alpha_8\}$ & 2 & 0\\
\hline
&$A_{1}^{3}$&$\{ \alpha_4,\alpha_6,\alpha_8\}$ & 2 & 0\\
\hline
&$A_{1}^{2}\times A_{2}$&$\{\alpha_1, \alpha_2,\alpha_4,\alpha_8\}$ &2  &$1\leq \rang(\mathfrak{q}_{\pi^{\prime}, \pi^{\prime\prime}}) \leq 2 $\\
\hline
&$A_{1}^{2}\times A_{2}$&$\{\alpha_1, \alpha_3,\alpha_6,\alpha_8\}$ & 2 & $1\leq \rang(\mathfrak{q}_{\pi^{\prime}, \pi^{\prime\prime}}) \leq 2$\\
\hline
&$A_{1}^{2}\times A_{2}$&$\{\alpha_3, \alpha_4,\alpha_6,\alpha_8\}$ &2  &$1\leq \rang(\mathfrak{q}_{\pi^{\prime}, \pi^{\prime\prime}}) \leq 2$ \\
\hline
&$A_{1}^{2}\times A_{2}$&$\{\alpha_2, \alpha_4,\alpha_6,\alpha_8\}$ & 2 &$ 1\leq \rang(\mathfrak{q}_{\pi^{\prime}, \pi^{\prime\prime}}) \leq 2$\\
\hline
&$A_{1}^{2}\times A_{2}$&$\{\alpha_1, \alpha_4,\alpha_7,\alpha_8\}$ & 2 & 0\\
\hline
&$A_{1}^{2}\times A_{2}$&$\{\alpha_1, \alpha_4,\alpha_5,\alpha_8\}$ & 2 & $1\leq \rang(\mathfrak{q}_{\pi^{\prime}, \pi^{\prime\prime}}) \leq 2$\\
\hline
&$A_{1}^{2}\times A_{2}$&$\{\alpha_1, \alpha_5,\alpha_6,\alpha_8\}$ & 2 & $1\leq \rang(\mathfrak{q}_{\pi^{\prime}, \pi^{\prime\prime}}) \leq 2$\\
\hline
&$A_{1}\times A_{2}^{2}$&$\{\alpha_1, \alpha_3,\alpha_5,\alpha_6,\alpha_8\}$ &2  & 2\\
\hline
&$A_{1}\times A_{2}^{2}$&$\{\alpha_1, \alpha_2,\alpha_4,\alpha_7,\alpha_8\}$ & 2 &$1\leq \rang(\mathfrak{q}_{\pi^{\prime}, \pi^{\prime\prime}}) \leq 2 $\\
\hline
&$A_{1}\times A_{2}^{2}$&$\{\alpha_1, \alpha_4,\alpha_5,\alpha_7,\alpha_8\}$ & 2 & $1\leq \rang(\mathfrak{q}_{\pi^{\prime}, \pi^{\prime\prime}}) \leq 2$\\
\hline
&$A_{1}\times A_{3}$&$\{\alpha_2, \alpha_3,\alpha_4,\alpha_8\}$ &2  & 2\\
\hline
&$A_{1}\times A_{3}$&$\{\alpha_3, \alpha_4,\alpha_5,\alpha_8\}$ & 2 & 2\\
\hline
&$A_{1}\times A_{3}$&$\{\alpha_2, \alpha_4,\alpha_5,\alpha_8\}$ & 2 &2 \\
\hline
&$A_{2}\times A_{3}$&$\{\alpha_2, \alpha_3,\alpha_4,\alpha_7,\alpha_8\}$ &2 & 2\\
\hline
&$A_{2}\times A_{3}$&$\{\alpha_3, \alpha_4,\alpha_5,\alpha_7,\alpha_8\}$ & 2 & 2\\
\hline
&$A_{2}\times A_{3}$&$\{\alpha_2, \alpha_4,\alpha_5,\alpha_7,\alpha_8\}$ & 2 & 2\\
\hline
&$A_{1}\times A_{4}$&$\{\alpha_1, \alpha_2,\alpha_3,\alpha_4,\alpha_8\}$ & 2 &2 \\
\hline
&$A_{1}\times A_{4}$&$\{\alpha_1, \alpha_3,\alpha_4,\alpha_5,\alpha_8\}$ & 2 & 2\\
\hline
&$A_{1}\times A_{4}$&$\{\alpha_3, \alpha_4,\alpha_5,\alpha_6,\alpha_8\}$ & 2 & 2\\
\hline
&$A_{1}\times A_{4}$&$\{\alpha_1, \alpha_5,\alpha_6,\alpha_7,\alpha_8\}$ & 2 & $1\leq \rang(\mathfrak{q}_{\pi^{\prime}, \pi^{\prime\prime}}) \leq 2$\\
\hline
&$A_{2}\times A_{4}$&$\{\alpha_1, \alpha_2,\alpha_3,\alpha_4,\alpha_7,\alpha_8\}$ & 2 & 2\\
\hline
&$A_{2}\times A_{4}$&$\{\alpha_1, \alpha_3,\alpha_4,\alpha_5,\alpha_7,\alpha_8\}$ & 2 & 2\\
\hline
&$A_{2}\times A_{4}$&$\{\alpha_1, \alpha_3,\alpha_5,\alpha_6,\alpha_7,\alpha_8\}$ & 2 &2 \\
\hline
&$A_{5}$&$\{\alpha_4,\alpha_5,\alpha_6,\alpha_7,\alpha_8\}$ & 2 & 2\\
\hline
&$A_{6}$&$\{\alpha_2, \alpha_4,\alpha_5,\alpha_6,\alpha_7,\alpha_8\}$ & 2 & 2\\
\hline
&$A_{6}$&$\{\alpha_3,\alpha_4,\alpha_5,\alpha_6,\alpha_7,\alpha_8\}$ & 2 & 2\\
\hline
\end{tabular}\\
\\
\\
\newpage
Cas o\`{u} le  cardinal de $\mathcal{K}(\pi^{\prime\prime})$ est supérieur ou égal à 4:\\
\\
\begin{tabular}{|l|l|l|l|l|l|l|c|c|c|}
\hline
&$Type$ & $ \pi^{\prime\prime}$ & $\ind(\mathfrak{q}_{\pi^{\prime}, \pi^{\prime\prime}})$ & $\rang(\mathfrak{q}_{\pi^{\prime}, \pi^{\prime\prime}})$ \\
\hline
&$A_{1}^{4}$& $\{\alpha_1,\alpha_4,\alpha_6, \alpha_8\}$ & 3 & 0\\
\hline
&$A_{1}^{3}\times A_{2}$& $\{\alpha_1,\alpha_{2}, \alpha_4,\alpha_6,\alpha_8\}$ &3 &$1\leq \rang(\mathfrak{q}_{\pi^{\prime}, \pi^{\prime\prime}}) \leq 3 $ \\
\hline
&$A_{1}^{2}\times A_{3}$&$\{\alpha_2,\alpha_{3}, \alpha_4,\alpha_6,\alpha_8\}$ & 3 &$1\leq \rang(\mathfrak{q}_{\pi^{\prime}, \pi^{\prime\prime}}) \leq 3$ \\
\hline
&$A_{1}^{2}\times A_{3}$&$\{\alpha_1,\alpha_{2}, \alpha_4,\alpha_5,\alpha_8\}$ & 3 &$1\leq \rang(\mathfrak{q}_{\pi^{\prime}, \pi^{\prime\prime}}) \leq 2 $\\
\hline
&$A_{1}\times A_{2}\times A_{3}$&$\{\alpha_1, \alpha_2,\alpha_4,\alpha_5,\alpha_7,\alpha_8\}$ &3  &$1\leq \rang(\mathfrak{q}_{\pi^{\prime}, \pi^{\prime\prime}}) \leq 3$ \\
\hline
&$A_{1}^{2}\times A_{4}$&$\{\alpha_1, \alpha_2,\alpha_3,\alpha_4,\alpha_6,\alpha_8\}$ & 3 & $1\leq \rang(\mathfrak{q}_{\pi^{\prime}, \pi^{\prime\prime}}) \leq 3$\\
\hline
&$A_{1}^{2}\times A_{4}$&$\{\alpha_1, \alpha_2,\alpha_4,\alpha_5,\alpha_6,\alpha_8\}$ & 3 & $1\leq \rang(\mathfrak{q}_{\pi^{\prime}, \pi^{\prime\prime}}) \leq 3$\\
\hline
&$A_{1}\times A_{5}$&$\{\alpha_1, \alpha_3,\alpha_4,\alpha_5,\alpha_6,\alpha_8\}$ & 3 & $1\leq \rang(\mathfrak{q}_{\pi^{\prime}, \pi^{\prime\prime}}) \leq 3$\\
\hline
&$A_{1}\times A_{5}$&$\{\alpha_1, \alpha_4, \alpha_5,\alpha_6,\alpha_7,\alpha_8\}$ & 3 &$1\leq \rang(\mathfrak{q}_{\pi^{\prime}, \pi^{\prime\prime}}) \leq 3$ \\
\hline
&$A_{1}\times A_{6}$&$\{\alpha_1, \alpha_2,\alpha_4,\alpha_5,\alpha_6,\alpha_7,\alpha_8\}$ & 3 & $1\leq \rang(\mathfrak{q}_{\pi^{\prime}, \pi^{\prime\prime}}) \leq 3$\\
\hline
&$D_{5}\times A_{1}$&$\{\alpha_1, \alpha_2,\alpha_3,\alpha_4,\alpha_5,\alpha_8\}$ &  4& 4\\
\hline
&$D_{5}\times A_{1}$&$\{\alpha_2, \alpha_3,\alpha_4,\alpha_5,\alpha_6,\alpha_8\}$ & 4 & 4\\
\hline
&$E_{6}\times A_{1}$&$\{\alpha_1,\alpha_2, \alpha_3,\alpha_4,\alpha_5,\alpha_6,\alpha_8\}$ & 4 &4 \\
\hline
&$D_{5}\times A_{2}$&$\{\alpha_1,\alpha_2,\alpha_3, \alpha_4,\alpha_5,\alpha_7,\alpha_8\}$ & 4 & 4\\
\hline
&$D_{7}$&$\{\alpha_2,\alpha_3, \alpha_4,\alpha_5,\alpha_6,\alpha_7,\alpha_8\}$ & 5 & 5\\
\hline
\end{tabular}\\
\\
\\
\\
\textbf{Cas o\`{u} $\pi^{\prime}$ est  non connexe et $\pi^{\prime\prime}$ est non connexe ou tel que $\mathcal{K}(\pi^{\prime\prime})$  est de cardinal inférieur ou égal à 4: }
\\

Dans ce cas, il est clair que si la sous-algèbre biparabolique $\mathfrak{q}_{\pi^{\prime}, \pi^{\prime\prime}}$ vérifie la condition $(\ast)$ alors $\pi^{\prime}$ ou $\pi^{\prime\prime}$ est de type $D_{4}\times A_{1}$, $D_{4}\times A_{2}$, $D_{5}\times A_{1}$, $D_{5}\times A_{2}$ ou $E_{6}\times A_{1}$.\\
\\
1. Supposons   $\pi^{\prime}$ de type $D_{4}\times A_{1}$:   $\pi^{\prime}=\{\alpha_2,\alpha_{3},\alpha_{4},\alpha_{5},\alpha_{7}\}$ ou $\pi^{\prime}=\{\alpha_2,\alpha_{3},\alpha_{4},\alpha_{5},\alpha_{8}\}$.\\
a. Soit  $\pi^{\prime}=\{\alpha_2,\alpha_{3},\alpha_{4},\alpha_{5},\alpha_{7}\}$. Ici, on a $\mathcal{E}_{\pi^{\prime}}=\{\alpha_{2}+\alpha_{3}+2\alpha_{4}+\alpha_{5},\alpha_{2},\alpha_{3},\alpha_{5},\alpha_7\}$:\\
\\
\begin{tabular}{|l|l|l|l|l|l|c|c|c|}
\hline
&$Type$ & $ \pi^{\prime\prime}$ & $\ind(\mathfrak{q}_{\pi^{\prime}, \pi^{\prime\prime}})$ & $\rang(\mathfrak{q}_{\pi^{\prime}, \pi^{\prime\prime}})$ \\
\hline
&$A_{1}^{4}$& $\{\alpha_1,\alpha_4,\alpha_6,\alpha_8\}$ & 1 & 0\\
\hline
&$A_{1}^{3}\times A_{2}$& $\{\alpha_1,\alpha_2, \alpha_4,\alpha_6,\alpha_8\}$ & 1 & 1\\
\hline
&$A_{1}^{2}\times A_{4}$&$\{\alpha_1,\alpha_2,\alpha_4,\alpha_5,\alpha_6,\alpha_8\}$ & 1 & 1\\
\hline
&$A_{1}^{2}\times A_{4}$&$\{\alpha_1, \alpha_2,\alpha_3,\alpha_4,\alpha_6,\alpha_8\}$ & 1 & 1\\
\hline
&$A_{1}\times A_{5}$&$\{\alpha_1,\alpha_4,\alpha_5,\alpha_6,\alpha_7,\alpha_8\}$ & 1 & 1\\
\hline
&$A_{1}\times A_{6}$&$\{\alpha_1, \alpha_2,\alpha_4,\alpha_5,\alpha_6,\alpha_7,\alpha_8\}$ & 1 & 0\\
\hline
\end{tabular}\\
\\
\\
b. Soit  $\pi^{\prime}=\{\alpha_2,\alpha_{3},\alpha_{4},\alpha_{5},\alpha_{8}\}$. Ici, on a $\mathcal{E}_{\pi^{\prime}}=\{\alpha_{2}+\alpha_{3}+2\alpha_{4}+\alpha_{5},\alpha_{2},\alpha_{3},\alpha_{5},\alpha_8\}$:\\
\\
\begin{tabular}{|l|l|l|l|l|l|c|c|c|}
\hline
&$Type$ & $ \pi^{\prime\prime}$ & $\ind(\mathfrak{q}_{\pi^{\prime}, \pi^{\prime\prime}})$ & $\rang(\mathfrak{q}_{\pi^{\prime}, \pi^{\prime\prime}})$ \\
\hline
&$A_{1}^{2}\times A_{3}$& $\{\alpha_1,\alpha_4,\alpha_6,\alpha_7,\alpha_8\}$ & 1 & 0\\
\hline
&$A_{1}\times A_{2}\times A_{3}$& $\{\alpha_1,\alpha_2, \alpha_4,\alpha_6,\alpha_7,\alpha_8\}$ & 1 & 1\\
\hline
&$A_{1}\times A_{5}$&$\{\alpha_1,\alpha_4,\alpha_5,\alpha_6,\alpha_7,\alpha_8\}$ & 1 & 0\\
\hline
&$A_{1}\times  A_{6}$&$\{\alpha_1, \alpha_2,\alpha_4,\alpha_5,\alpha_6,\alpha_7,\alpha_8\}$ & 1 & 1\\
\hline
&$ A_{3}\times A_{4}$&$\{\alpha_1, \alpha_2,\alpha_3,\alpha_4,\alpha_6,\alpha_7,\alpha_8\}$ & 1 & 1\\
\hline
\end{tabular}\\
\\
\\
2. Supposons   $\pi^{\prime}$ de type $D_{4}\times A_{2}$:   $\pi^{\prime}=\{\alpha_2,\alpha_{3},\alpha_{4},\alpha_{5},\alpha_{7},\alpha_8\}$. Ici, on a $\mathcal{E}_{\pi^{\prime}}=\{\alpha_{2}+\alpha_{3}+2\alpha_{4}+\alpha_{5},\alpha_{2},\alpha_{3},\alpha_{5},\alpha_7+\alpha_8\}$:\\
\\
\begin{tabular}{|l|l|l|l|l|l|c|c|c|}
\hline
&$Type$ & $ \pi^{\prime\prime}$ & $\ind(\mathfrak{q}_{\pi^{\prime}, \pi^{\prime\prime}})$ & $\rang(\mathfrak{q}_{\pi^{\prime}, \pi^{\prime\prime}})$ \\
\hline
&$A_{1}^{4}$& $\{\alpha_1,\alpha_4,\alpha_6,\alpha_8\}$ & 1 & 0\\
\hline
&$A_{1}^{3}\times A_{2}$& $\{\alpha_1,\alpha_2, \alpha_4,\alpha_6,\alpha_8\}$ &1  & 0\\
\hline
&$A_{1}^{2}\times A_{3}$&$\{\alpha_1,\alpha_4,\alpha_6,\alpha_7,\alpha_8\}$ & 1 & 0\\
\hline
&$A_{1}^{2}\times A_{4}$&$\{\alpha_1, \alpha_2,\alpha_4,\alpha_5,\alpha_6,\alpha_8\}$ & 1 & 1\\
\hline
&$A_{1}^{2}\times A_{4}$&$\{\alpha_1, \alpha_2,\alpha_3,\alpha_4,\alpha_6,\alpha_8\}$ & 1 &1 \\
\hline
&$A_{1}\times A_{2}\times A_{3}$&$\{\alpha_1, \alpha_2,\alpha_4,\alpha_6,\alpha_7,\alpha_8\}$ &1  & 1\\
\hline
&$ A_{1}\times A_{5}$&$\{\alpha_1, \alpha_4,\alpha_5,\alpha_6,\alpha_7,\alpha_8\}$ & 1 & 1\\
\hline
&$A_{1}\times A_{6}$&$\{\alpha_1,\alpha_2,\alpha_4,\alpha_5,\alpha_6,\alpha_7,\alpha_8\}$ & 1 & 1\\
\hline
&$A_{3}\times A_{4}$&$\{\alpha_1, \alpha_2,\alpha_3,\alpha_4,\alpha_6,\alpha_7,\alpha_8\}$ & 1 & 1\\
\hline
\end{tabular}\\
\\
\\
3. Supposons   $\pi^{\prime}$ de type $D_{5}\times A_{1}$:   $\pi^{\prime}=\{\alpha_1,\alpha_{2},\alpha_{3},\alpha_{4},\alpha_5,\alpha_{7}\}$ ou $\pi^{\prime}=\{\alpha_1,\alpha_2,\alpha_{3},\alpha_{4},\alpha_{5},\alpha_{8}\}$ ou $\pi^{\prime}=\{\alpha_2,\alpha_{3},\alpha_{4},\alpha_{5},\alpha_6,\alpha_{8}\}$ .\\
a. Soit  $\pi^{\prime}=\{\alpha_1,\alpha_2,\alpha_{3},\alpha_{4},\alpha_{5},\alpha_{7}\}$. Ici, on a $\mathcal{E}_{\pi^{\prime}}=\{\alpha_1+\alpha_{2}+2\alpha_{3}+2\alpha_{4}+\alpha_{5},\alpha_{2}+\alpha_{4}+\alpha_{5},\alpha_1,\alpha_4,\alpha_7\}$:\\
\\
\begin{tabular}{|l|l|l|l|l|l|c|c|c|}
\hline
&$Type$ & $ \pi^{\prime\prime}$ & $\ind(\mathfrak{q}_{\pi^{\prime}, \pi^{\prime\prime}})$ & $\rang(\mathfrak{q}_{\pi^{\prime}, \pi^{\prime\prime}})$ \\
\hline
&$A_{1}^{4}$& $\{\alpha_2,\alpha_3,\alpha_6,\alpha_8\}$ & 1 & 0\\
\hline
&$A_{1}^{3}\times A_{2}$& $\{\alpha_1,\alpha_2, \alpha_3,\alpha_6,\alpha_8\}$ &  1& 1\\
\hline
&$A_{1}^{3}\times A_{2}$& $\{\alpha_2,\alpha_3, \alpha_5,\alpha_6,\alpha_8\}$ &  1& 0\\
\hline
&$A_{1}^{2}\times A_{2}^{2}$& $\{\alpha_1,\alpha_2, \alpha_3,\alpha_5,\alpha_6,\alpha_8\}$ & 1 & 1\\
\hline
&$A_{1}^{2}\times A_{4}$&$\{\alpha_1,\alpha_2,\alpha_3,\alpha_4,\alpha_6,\alpha_8\}$ &1  &1 \\
\hline
&$A_{1}^{2}\times A_{4}$&$\{\alpha_2,\alpha_3,\alpha_5,\alpha_6,\alpha_7,\alpha_8\}$ & 1 & 0\\
\hline
&$A_{1}\times A_{2}\times A_{4}$& $\{\alpha_1,\alpha_2, \alpha_3,\alpha_5,\alpha_6,\alpha_7,\alpha_8\}$ & 1 & 1\\
\hline
&$A_{7}$&$\{\alpha_1, \alpha_3,\alpha_4,\alpha_5,\alpha_6,\alpha_7,\alpha_8\}$ & 1 &0 \\
\hline
\end{tabular}\\
\\
\\
b. Soit  $\pi^{\prime}=\{\alpha_1,\alpha_2,\alpha_{3},\alpha_{4},\alpha_{5},\alpha_{8}\}$. Ici, on a $\mathcal{E}_{\pi^{\prime}}=\{\alpha_1+\alpha_{2}+2\alpha_{3}+2\alpha_{4}+\alpha_{5},\alpha_{2}+\alpha_{4}+\alpha_{5},\alpha_1,\alpha_4,\alpha_8\}$:\\
\\
\begin{tabular}{|l|l|l|l|l|l|c|c|c|}
\hline
&$Type$ & $ \pi^{\prime\prime}$ & $\ind(\mathfrak{q}_{\pi^{\prime}, \pi^{\prime\prime}})$ & $\rang(\mathfrak{q}_{\pi^{\prime}, \pi^{\prime\prime}})$ \\
\hline
&$A_{1}^{2}\times A_{3}$& $\{\alpha_2,\alpha_3,\alpha_6,\alpha_7,\alpha_8\}$ & 1 & 0\\
\hline
&$A_{1}^{2}\times  A_{3}$& $\{\alpha_2, \alpha_3,\alpha_5,\alpha_6,\alpha_7\}$ & 1 & 0\\
\hline
&$A_{1}\times A_{2}\times A_{3}$& $\{\alpha_1,\alpha_2,\alpha_3, \alpha_6,\alpha_7,\alpha_8\}$ &1  & 1\\
\hline
&$A_{1}\times A_{2}\times  A_{3}$& $\{\alpha_1,\alpha_2, \alpha_3,\alpha_5,\alpha_6,\alpha_7\}$ & 1 & 1\\
\hline
&$A_{3}\times A_{4}$&$\{\alpha_1,\alpha_2,\alpha_3,\alpha_4,\alpha_6,\alpha_7,\alpha_8\}$ &1  & 1\\
\hline
&$A_{7}$&$\{\alpha_1, \alpha_3,\alpha_4,\alpha_5,\alpha_6,\alpha_7,\alpha_8\}$ &1  & 1\\
\hline
\end{tabular}\\
\\
\\
c. Soit  $\pi^{\prime}=\{\alpha_2,\alpha_{3},\alpha_{4},\alpha_{5},\alpha_{6}\,\alpha_8\}$. Ici, on a $\mathcal{E}_{\pi^{\prime}}=\{\alpha_2+\alpha_{3}+2\alpha_{4}+2\alpha_{5}+\alpha_{6},\alpha_{2}+\alpha_{3}+\alpha_{4},\alpha_4,\alpha_6,\alpha_8\}$:\\
\\
\begin{tabular}{|l|l|l|l|l|l|c|c|c|}
\hline
&$Type$ & $ \pi^{\prime\prime}$ & $\ind(\mathfrak{q}_{\pi^{\prime}, \pi^{\prime\prime}})$ & $\rang(\mathfrak{q}_{\pi^{\prime}, \pi^{\prime\prime}})$ \\
\hline
&$A_{1}^{4}$& $\{\alpha_1,\alpha_2,\alpha_5,\alpha_7\}$ & 1 & 0\\
\hline
&$A_{1}^{3}\times A_{2}$& $\{\alpha_1,\alpha_2, \alpha_3,\alpha_5,\alpha_7\}$ &1  & 0\\
\hline
&$A_{1}^{3}\times A_{2}$& $\{\alpha_1,\alpha_2, \alpha_5,\alpha_7,\alpha_8\}$ & 1 & 0\\
\hline
&$A_{1}^{2}\times A_{2}^{2}$& $\{\alpha_1,\alpha_2, \alpha_3,\alpha_5,\alpha_7,\alpha_8\}$ &1  &0 \\
\hline
&$A_{1}^{2}\times A_{3}$&$\{\alpha_1,\alpha_2,\alpha_6,\alpha_7,\alpha_8\}$ & 1 & 1\\
\hline
&$A_{1}^{2}\times A_{4}$&$\{\alpha_1,\alpha_2,\alpha_5,\alpha_6,\alpha_7,\alpha_8\}$ & 1 & 0\\
\hline
&$A_{1}\times A_{2}\times A_{3}$& $\{\alpha_1,\alpha_2, \alpha_3,\alpha_6,\alpha_7,\alpha_8\}$ & 1 & 1\\
\hline
&$A_{1}\times A_{2}\times A_{3}$& $\{\alpha_1,\alpha_2, \alpha_4,\alpha_6,\alpha_7,\alpha_8\}$ & 1 & 1\\
\hline
&$A_{1}\times A_{2}\times A_{4}$& $\{\alpha_1,\alpha_2, \alpha_3,\alpha_5,\alpha_6,\alpha_7,\alpha_8\}$ & 1 & 0\\
\hline
&$A_{1}\times A_{6}$& $\{\alpha_1,\alpha_2, \alpha_4,\alpha_5,\alpha_6,\alpha_7,\alpha_8\}$ & 1 & 1\\
\hline
&$A_{3}\times A_{3}$& $\{\alpha_1,\alpha_3, \alpha_4,\alpha_6,\alpha_7,\alpha_8\}$ &  1&1 \\
\hline
&$A_{3}\times A_{4}$& $\{\alpha_1,\alpha_2,\alpha_3, \alpha_4,\alpha_6,\alpha_7,\alpha_8\}$ & 1 & 1\\
\hline
\end{tabular}\\
\\
\\
3. Supposons maintenant  $\pi^{\prime}$ de type $D_{5}\times A_{2}$. Ici, on a  $\pi^{\prime}=\{\alpha_1,\alpha_{2},\alpha_{3},\alpha_{4},\alpha_5,\alpha_{7},\alpha_8\}$ et $\mathcal{E}_{\pi^{\prime}}=\{\alpha_1+\alpha_{2}+2\alpha_{3}+2\alpha_{4}+\alpha_{5},\alpha_{2}+\alpha_{4}+\alpha_{5},\alpha_1,\alpha_4,\alpha_7+\alpha_8\}$:\\
\\
\begin{tabular}{|l|l|l|l|l|l|c|c|c|}
\hline
&$Type$ & $ \pi^{\prime\prime}$ & $\ind(\mathfrak{q}_{\pi^{\prime}, \pi^{\prime\prime}})$ & $\rang(\mathfrak{q}_{\pi^{\prime}, \pi^{\prime\prime}})$ \\
\hline
&$A_{1}^{4}$& $\{\alpha_2,\alpha_3,\alpha_6,\alpha_8\}$ & 1 & 0\\
\hline
&$A_{1}^{3}\times A_{2}$& $\{\alpha_1,\alpha_2, \alpha_3,\alpha_6,\alpha_8\}$ & 1 & 1\\
\hline
&$A_{1}^{3}\times A_{2}$& $\{\alpha_2,\alpha_3, \alpha_5,\alpha_6,\alpha_8\}$ &  1& 0\\
\hline
&$A_{1}^{2}\times A_{2}^{2}$& $\{\alpha_1,\alpha_2, \alpha_3,\alpha_5,\alpha_6,\alpha_8\}$ & 1 &1 \\
\hline
&$A_{1}^{2}\times A_{3}$&$\{\alpha_2,\alpha_3,\alpha_6,\alpha_7,\alpha_8\}$ & 1 &0 \\
\hline
&$A_{1}^{2}\times A_{3}$&$\{\alpha_2,\alpha_3,\alpha_5,\alpha_6,\alpha_7\}$ & 1 & 0\\
\hline
&$A_{1}^{2}\times A_{4}$& $\{\alpha_1,\alpha_2, \alpha_3,\alpha_4,\alpha_6,\alpha_8\}$ & 1 & 1\\
\hline
&$A_{1}^{2}\times A_{4}$& $\{\alpha_2,\alpha_3, \alpha_5,\alpha_6,\alpha_7,\alpha_8\}$ & 1 & 0\\
\hline
&$A_{1}\times A_{2}\times A_{3}$&$\{\alpha_1,\alpha_2,\alpha_3,\alpha_5,\alpha_6,\alpha_7\}$ & 1 & 1\\
\hline
&$A_{1}\times A_{2}\times A_{3}$&$\{\alpha_1,\alpha_2,\alpha_3,\alpha_6,\alpha_7,\alpha_8\}$ & 1 & 1\\
\hline
&$A_{1}\times A_{2}\times A_{4}$&$\{\alpha_1,\alpha_2,\alpha_3,\alpha_5,\alpha_6,\alpha_7,\alpha_8\}$ & 1 & 1\\
\hline
&$A_{3}\times A_{4}$&$\{\alpha_1,\alpha_2,\alpha_3,\alpha_4,\alpha_6,\alpha_7,\alpha_8\}$ & 1 & 1\\
\hline
&$A_{7}$&$\{\alpha_1, \alpha_3,\alpha_4,\alpha_5,\alpha_6,\alpha_7,\alpha_8\}$ & 1 & 1\\
\hline
\end{tabular}\\
\\
\\
4. Supposons maintenant  $\pi^{\prime}$ de type $E_{6}\times A_{1}$. Ici, on a  $\pi^{\prime}=\{\alpha_1,\alpha_{2},\alpha_{3},\alpha_{4},\alpha_5,\alpha_{6},\alpha_8\}$ et $\mathcal{E}_{\pi^{\prime}}=\{\alpha_1+2\alpha_{2}+2\alpha_{3}+3\alpha_{4}+2\alpha_{5}+\alpha_6,\alpha_{1}+\alpha_{3}+\alpha_{4}+\alpha_{5}+\alpha_{6},
\alpha_{3}+\alpha_{4}+\alpha_{5},\alpha_4,\alpha_8\}$:\\
\\
\begin{tabular}{|l|l|l|l|l|l|c|c|c|}
\hline
&$Type$ & $ \pi^{\prime\prime}$ & $\ind(\mathfrak{q}_{\pi^{\prime}, \pi^{\prime\prime}})$ & $\rang(\mathfrak{q}_{\pi^{\prime}, \pi^{\prime\prime}})$ \\
\hline
&$A_{1}^{4}$& $\{\alpha_1,\alpha_2,\alpha_5,\alpha_7\}$ & 1 & 0\\
\hline
&$A_{1}^{3}\times A_{2}$& $\{\alpha_1,\alpha_2, \alpha_5,\alpha_7,\alpha_8\}$ & 1 & 0\\
\hline
&$A_{1}^{3}\times A_{2}$& $\{\alpha_1,\alpha_2, \alpha_3,\alpha_5,\alpha_7\}$ & 1 & 0\\
\hline
&$A_{1}^{2}\times A_{2}^{2}$& $\{\alpha_1,\alpha_2, \alpha_3,\alpha_5,\alpha_7,\alpha_8\}$ &1  &0 \\
\hline
&$A_{1}^{2}\times A_{3}$&$\{\alpha_2,\alpha_3,\alpha_6,\alpha_7,\alpha_8\}$ & 1 & 0\\
\hline
&$A_{1}^{2}\times A_{3}$&$\{\alpha_2,\alpha_3,\alpha_5,\alpha_6,\alpha_7\}$ &1  & 0\\
\hline
&$A_{1}\times A_{2}\times A_{3}$& $\{\alpha_1,\alpha_2, \alpha_3,\alpha_6,\alpha_7,\alpha_8\}$ & 1 &0 \\
\hline
&$A_{1}\times A_{2}\times A_{3}$&$\{\alpha_1,\alpha_2,\alpha_3,\alpha_5,\alpha_6,\alpha_7\}$ &1  & 0\\
\hline
&$A_{1}^{2}\times A_{4}$&$\{\alpha_1,\alpha_2,\alpha_5,\alpha_6,\alpha_7,\alpha_8\}$ &  1& 0\\
\hline
&$A_{1}\times A_{2}\times A_{4}$&$\{\alpha_1,\alpha_2,\alpha_3,\alpha_5,\alpha_6,\alpha_7,\alpha_8\}$ & 1 & 0\\
\hline
&$A_{1}\times A_{5}$&$\{\alpha_1,\alpha_2,\alpha_4,\alpha_5,\alpha_6,\alpha_7\}$ & 1 & 1\\
\hline
&$A_{3}\times A_{3}$&$\{\alpha_1,\alpha_3,\alpha_4,\alpha_6,\alpha_7,\alpha_8\}$ & 1 & 1\\
\hline
&$A_{3}\times A_{4}$&$\{\alpha_1,\alpha_2,\alpha_3,\alpha_4,\alpha_6,\alpha_7,\alpha_8\}$ & 1 & 0\\
\hline
&$A_{7}$&$\{\alpha_1, \alpha_3,\alpha_4,\alpha_5,\alpha_6,\alpha_7,\alpha_8\}$ & 1 & 1\\
\hline
\end{tabular}\\

\bibliographystyle{smfplain}
\bibliography{Stability for biparabolic subalgebras of simple Lie algebras}

\providecommand{\bysame}{\leavevmode ---\ }
\providecommand{\og}{``}
\providecommand{\fg}{''}
\providecommand{\smfandname}{et}
\providecommand{\smfedsname}{\'eds.}
\providecommand{\smfedname}{\'ed.}
\providecommand{\smfmastersthesisname}{M\'emoire}
\providecommand{\smfphdthesisname}{Th\`ese}
\begin{thebibliography}{10}

\bibitem{Ammari}
{\scshape K.~Ammari} -- {\og Stabilit\'{e} des sous-alg\`{e}bres paraboliques
  de $\mathfrak{so}(n)$\fg}, \emph{J. of Lie Theory} \textbf{24} (2013),
  p.~97--122.

\bibitem{Ammari2}
\bysame , {\og Stabilit\'{e} des sous-alg\`{e}bres paraboliques des algèbres de
  {L}ie simples exceptionnelles\fg}, \emph{Bull. Sci. Math} \textbf{138}
  (2014), no.~05, p.~614--625.

\bibitem{Anne+Karin}
{\scshape K.~Baur {\normalfont \smfandname} A.~Moreau} -- {\og Quasi-reductive
  (bi)parabolic subalgebras in reductive {L}ie algebras\fg}, \emph{Ann. Inst.
  Fourier (Grenoble)} \textbf{61} (2011), no.~2, p.~417--451.

\bibitem{Bourbaki}
{\scshape N.~Bourbaki} -- \emph{Groupes et algèbres de {L}ie. {C}hapitres 4, 5
  et 6}, Masson, Paris, 1981.

\bibitem{Dixmier}
{\scshape J.~Dixmier} -- \emph{Enveloping algebras}, Graduate Studies in
  Mathematics, vol.~11, American Mathematical Society, Providence, RI, 1996,
  Revised reprint of the 1977 translation.

\bibitem{Djebali}
{\scshape N.~Djebali} -- {\og Sous-groupes réductifs canoniques des
  sous-groupes biparaboliques de so(n,$\mathbb{C}$) ou so(p,q) dont l'algèbre
  de lie est quasi-réductive\fg}, \emph{J. of Lie Theory} \textbf{28} (2018),
  p.~443--477.

\bibitem{duflo-1982}
{\scshape M.~Duflo} -- {\og {Th\'{e}orie de Mackey pour les groupes de {L}ie
  alg\'{e}briques}\fg}, \emph{Acta. Math.} \textbf{149} (1982), p.~153--213.

\bibitem{DKT}
{\scshape M.~Duflo, M.~S. Khalgui {\normalfont \smfandname} P.~Torasso} -- {\og
  Algèbres de {L}ie quasi-réductives\fg}, \emph{Transformation Groups}
  \textbf{17} (2012), no.~2, p.~417--470.

\bibitem{duflo-vergne-1969}
{\scshape M.~Duflo {\normalfont \smfandname} M.~Vergne} -- {\og {Une
  propri\'et\'e de la repr\'esentation co-adjointe d'une alg\`ebre de
  {L}ie}\fg}, \emph{C. R. Acad. Sci. Paris S\'er. A-B} \textbf{268} (1969),
  p.~A583--A585.

\bibitem{joseph1977}
{\scshape A.~Joseph} -- {\og A preparation theorem for the prime spectrum of a
  semisimple {L}ie algebra\fg}, \emph{J. Algebra} \textbf{48} (1977), no.~2,
  p.~241--289.

\bibitem{MR2264140}
\bysame , {\og On semi-invariants and index for biparabolic (seaweed) algebras.
  {I}\fg}, \emph{J. Algebra} \textbf{305} (2006), no.~1, p.~487--515.

\bibitem{Kosmann1974}
{\scshape Y.~Kosmann {\normalfont \smfandname} S.~Sternberg} -- {\og
  Conjugaison des sous-alg\`ebres d'isotropie\fg}, \emph{C. R. Acad. Sci. Paris
  S\'er. A} \textbf{279} (1974), p.~777--779.

\bibitem{kostant2011cascade}
{\scshape B.~Kostant} -- {\og The cascade of orthogonal roots and the coadjoint
  structure of the nilradical of a {B}orel subgroup of a semisimple {L}ie
  group\fg}, \emph{preprint arXiv:1101.5382} (2011).

\bibitem{kostant2012center}
\bysame , {\og Center of u(n), {C}ascade of {O}rthogonal {R}oots, and a
  {C}onstruction of {L}ipsman-{W}olf\fg}, \emph{preprint arXiv:1201.4494}
  (2012).

\bibitem{Panyushev2005}
{\scshape D.~I. Panyushev} -- {\og An extension of {R}a\"\i s' theorem and
  seaweed subalgebras of simple {L}ie algebras\fg}, \emph{Ann. Inst. Fourier
  (Grenoble)} \textbf{55} (2005), no.~3, p.~693--715.

\bibitem{Panyushev2017}
{\scshape D.~I. Panyushev {\normalfont \smfandname} O.~Yakimova} -- {\og On
  seaweed subalgebras and meander graphs in type {D}\fg},
  \emph{https://arxiv.org/pdf/1702.07879.pdf} (2017).

\bibitem{Tauvel-Yu2004}
{\scshape P.~Tauvel {\normalfont \smfandname} R.~W.~T. Yu} -- {\og Indice et
  formes lin\'eaires stables dans les alg\`ebres de {L}ie\fg}, \emph{J.
  Algebra} \textbf{273} (2004), no.~2, p.~507--516.

\bibitem{Tauvel-Yu}
\bysame , \emph{Lie algebras and algebraic groups}, Springer Monographs in
  Mathematics, Springer-Verlag, Berlin, 2005.

\bibitem{Tauvel-Yu2005}
\bysame , {\og Sur l'indice de certaines alg\`ebres de {L}ie\fg}, \emph{Ann.
  Inst. Fourier (Grenoble)} \textbf{54} (2005), no.~6, p.~1793--1810.

\end{thebibliography}

\end{document}